\documentclass{amsart}

\usepackage{amsmath,amscd}
\usepackage{amssymb}

\usepackage{times}

\usepackage[T1]{fontenc}
\usepackage{graphicx}
\usepackage[svgnames]{xcolor}
\usepackage{framed}

\usepackage[all]{xy}

\usepackage{enumitem,kantlipsum}

\usepackage{hyperref}
\hypersetup{urlcolor=blue, citecolor=blue, linkcolor=blue,hypertexnames=false}

\author{Liran Shaul}

\address{Fakult\"at f\"ur Mathematik\\ 
Universit\"at Bielefeld\\ 
33501 Bielefeld\\ 
Germany.}

\curraddr{Department of Mathematics, Ben Gurion University of the Negev, Beer Sheva 84105, Israel}

\email{shlir@post.bgu.ac.il}
\newtheorem{thm}[equation]{Theorem}
\newtheorem{cor}[equation]{Corollary}
\newtheorem{prop}[equation]{Proposition}
\newtheorem{lem}[equation]{Lemma}

\theoremstyle{definition}
\newtheorem{dfn}[equation]{Definition}
\newtheorem{rem}[equation]{Remark}
\newtheorem{exa}[equation]{Example}

\newtheorem{notation}[equation]{Notation}

\newtheorem{defx}{Definition}

\newcommand{\inj}{\hookrightarrow}

\newcommand{\opn}{\operatorname}
\newcommand{\cat}[1]{\operatorname{\mathsf{#1}}}

\newcommand{\mfrak}[1]{\mathfrak{#1}}

\newcommand{\mrm}[1]{\mathrm{#1}}

\renewcommand{\k}{\Bbbk}

\renewcommand{\a}{\mfrak{a}}

\newcommand{\injdim}{\operatorname{inj\,dim}}

\newcommand{\amp}{\operatorname{amp}}
\newcommand{\Ext}{\operatorname{Ext}}
\newcommand{\Tor}{\operatorname{Tor}}
\newcommand{\RHom}{\mrm{R}\opn{Hom}}

\newcommand{\INJ}{\cat{Inj}}
\newcommand{\bflatdim}{\operatorname{bfl\,dim}}
\newcommand{\uflatdim}{\operatorname{ufl\,dim}}

\newcommand{\p}{\bar{\mfrak{p}}}
\newcommand{\m}{\bar{\mfrak{m}}}
\renewcommand{\a}{\bar{\mfrak{a}}}
\numberwithin{equation}{section} 


\begin{document}

\title{Injective DG-modules over non-positive DG-rings}

\begin{abstract}
Let $A$ be an associative non-positive differential graded ring.
In this paper we make a detailed study of a category $\INJ(A)$ of left DG-modules over $A$ which generalizes the category of injective modules over a ring.
We give many characterizations of this category, generalizing the theory of injective modules, and 
prove a derived version of the Bass-Papp theorem: the category $\INJ(A)$ is closed in the derived category $\cat{D}(A)$ under arbitrary direct sums if and only if 
the ring $\mrm{H}^0(A)$ is left noetherian and for every $i<0$ the left $\mrm{H}^0(A)$-module $\mrm{H}^i(A)$ is finitely generated.
Specializing further to the case of commutative noetherian DG-rings, 
we generalize the Matlis structure theory of injectives to this context.
As an application, we obtain a concrete version of Grothendieck's local duality theorem over commutative noetherian local DG-rings.
\end{abstract}
\maketitle

\setcounter{tocdepth}{1}

\tableofcontents

\section{Introduction}

The notion of an injective module is one of the most fundamental notions in homological algebra.
In the theory of noetherian rings it became particularly important ever since Matlis foundational paper \cite{Ma},
which completely classified the injective modules in the noetherian case. 
The aim of this paper is to study this notion in higher algebra.

Our model for higher rings is given by non-positive DG-rings $A = \bigoplus_{n=-\infty}^0 A^n$,
with differentials of degree $+1$. 
This is a model for higher rings because by \cite[Theorem 1.1]{SS},
the category of non-positive DG-rings is Quillen equivalent to the category of simplicial rings.
Given a such non-positive DG-ring $A$,
attached to it is its derived category $\cat{D}(A)$ of left DG-modules over $A$.
This is a triangulated category, and the basic philosophy of study in the theory of DG-rings is that one should study them by studying $\cat{D}(A)$.

The first task of this paper is to identify the class of left DG-modules in $\cat{D}(A)$ which generalizes the class of left injective modules over an ordinary ring.
Over rings there are many different characterizations of injective modules,
all of which are equivalent.

That the same equivalence holds over non-positive DG-rings is the first main result of this paper.
Thus, for various characterizations of injectivity,
we explain how to restate them in the language of triangulated categories, and show that they are in fact equivalent over a non-positive DG-ring.
The result states:

\begin{thm}\label{thmI}
Let $A$ be a non-positive DG-ring.
Then the following are equivalent for a DG-module $I \in \cat{D}(A)$:
\begin{enumerate}[wide, labelwidth=!, labelindent=0pt]
\item Either $I = 0$ or $\inf\{i \mid \mrm{H}^i(I) \ne 0\} = 0$, and
\[
\inf \{ n \in \mathbb{Z} \mid \Ext^i_A(N,I) = 0 \text{ for any } N \in \cat{D}^{\mrm{b}}(A) \text{ and any } i>n-\inf N\} = 0.
\]
That is, the injective dimension of $I$ with respect to bounded DG-modules is $0$.
\item Either $I = 0$ or $\inf\{i \mid \mrm{H}^i(I) \ne 0\} = 0$, and
\[
\inf \{ n \in \mathbb{Z} \mid \Ext^i_A(N,I) = 0 \text{ for any } N \in \cat{D}^{+}(A) \text{ and any } i>n-\inf N\} = 0.
\]
That is, the injective dimension of $I$ with respect to unbounded DG-modules is $0$.
\item $I \in \cat{D}^{+}(A)$, and for any ring $B$ and any map of DG-rings $A \to B$,
there is an injective $B$-module $J$ and an isomorphism
\[
\mrm{R}\opn{Hom}_A(B,I) \cong J
\]
in $\cat{D}(B)$.
\item For any $n \in \mathbb{Z}$ and any $M \in \cat{D}(A)$ there is an isomorphism
\[
\mrm{H}^n \left(\mrm{R}\opn{Hom}_A(M,I)\right) \cong \opn{Hom}_{\mrm{H}^0(A)}(\mrm{H}^{-n}(M),\mrm{H}^0(I))
\]
in $\opn{Mod}\left(\mrm{H}^0(Z(A))\right)$ which is functorial in $M$. Here, $Z(A)$ is the center of $A$.
\item For any morphism $f:M \to N$ in $\cat{D}(A)$ such that $\mrm{H}^0(f):\mrm{H}^0(M) \to \mrm{H}^0(N)$ is an injective map,
the map 
\[
\opn{Hom}_{\cat{D}(A)}(f,I): \opn{Hom}_{\cat{D}(A)}(N,I) \to \opn{Hom}_{\cat{D}(A)}(M,I)
\]
is surjective.
\item For any morphism $f:I \to M$ in $\cat{D}(A)$ such that $\mrm{H}^0(f): \mrm{H}^0(I) \to \mrm{H}^0(M)$ is an injective map,
the morphism $f$ is a split monomorphism.
\end{enumerate}
\end{thm}

\noindent
We prove this result at the end of Section \ref{sec:eqv}.

For a non-positive DG-ring $A$,
we denote by $\INJ(A)$ the full subcategory of $\cat{D}(A)$ consisting of DG-modules satisfying the equivalent conditions of Theorem \ref{thmI}.
Note that if $A$ is an ordinary ring then $\INJ(A)$ is equivalent to the full subcategory of $\opn{Mod}(A)$ of injective left $A$-modules. Our second main result shows that $\INJ(A)$ is always equivalent to the category of injective modules over some ring. Precisely, we show:

\begin{thm}
Let $A$ be a non-positive DG-ring.
For any $I \in \INJ(A)$,
the left $\mrm{H}^0(A)$-module $\mrm{H}^0(I)$ is injective,
and the functor
\[
\mrm{H}^0(-):\INJ(A) \to \INJ(\mrm{H}^0(A))
\]
is an equivalence of categories.
\end{thm}

\noindent
This is repeated as Theorem \ref{thm:eqv} in the body of the paper.

Having established a category of DG-modules analogues to the category of injective modules over an ordinary ring,
we wish to follow Matlis and study it in the noetherian case.
There are various noetherian conditions imposed in the literature on non-positive DG-rings.
Instead of choosing one of these conditions,
we have a better alternative: recall that the Bass-Papp Theorem states that
a ring $A$ is left-noetherian if and only if the category of injective left $A$-modules is closed under arbitrary direct sums.
Analogously, we prove in Theorem \ref{thm:bass-papp}:

\begin{thm}\label{mnthm:noeth}
Let $A$ be a non-positive DG-ring.
Then the category $\INJ(A)$ is closed under arbitrary direct sums if and only if
the ring $\mrm{H}^0(A)$ is a left noetherian ring and for each $i < 0$, 
the left $\mrm{H}^0(A)$-module $\mrm{H}^i(A)$ is finitely generated.
\end{thm}

As far as we know, this is the first non-trivial characterization of noetherianity in higher algebra.
In view of this result, we say that a non-positive DG-ring is left noetherian if it satisfies the equivalent conditions of Theorem \ref{mnthm:noeth}.

As a corollary of Theorem \ref{mnthm:noeth}, 
we construct in Corollary \ref{cor:cogen} over a non-positive left noetherian DG-ring, a DG-module generalizing the minimal injective cogenerator over a left noetherian ring.  
We show that it cogenerates $\cat{D}(A)$.

In the final Section \ref{sec:comm} we specialize our study further to the category of non-positive commutative noetherian DG-rings.
Our first main result of that section states:
\begin{thm}
Let $A$ be a commutative noetherian DG-ring.
\begin{enumerate}
\item There is a bijection between the indecomposable elements of $\INJ(A)$ and elements $\p \in \opn{Spec}(\mrm{H}^0(A))$.
We denote by $E(A,\p)$ the indecomposable element of $\INJ(A)$ corresponding to $\p \in \opn{Spec}(\mrm{H}^0(A))$.
\item Every element of $\INJ(A)$ is a direct sum of DG-modules of the form $E(A,\p)$ for $\p \in \opn{Spec}(\mrm{H}^0(A))$.
\item Given $\p \in \opn{Spec}(\mrm{H}^0(A))$, denoting by $A_{\p}$ the localized DG-ring of $A$ at $\p$,
the DG-module $E(A,\p)$ has a structure of a DG-module over $A_{\p}$.
\item Given $\p \in \opn{Spec}(\mrm{H}^0(A))$, denoting by $\widehat{A_{\p}}$ the derived $\p$-adic completion of $A_{\p}$,
the DG-module $E(A,\p)$ has a structure of a DG-module over $\widehat{A_{\p}}$.
\item There is an isomorphism $\mrm{R}\opn{Hom}_A(E(A,\p),E(A,\p)) \cong \widehat{A_{\p}}$ in the homotopy category of DG-rings.
\end{enumerate}
\end{thm}

The last main result of this paper specializes further to local DG-rings.
A commutative noetherian DG-ring $A$ is called local if $\mrm{H}^0(A)$ is a local ring.
In that case, if $\m$ is the maximal ideal of $\mrm{H}^0(A)$ and $\k$ is the residue field of $\mrm{H}^0(A)$,
one says that $(A,\m,\k)$ is a local noetherian DG-ring. 
We denote by $\mrm{H}_{\m}^n(-)$ the $n$-th local cohomology over $A$ functor with respect to $\m$.

Our final result generalizes Grothendieck's local duality to the commutative noetherian local DG-setting.
Its statement does not mention the category $\INJ(A)$, but the theory explained above is used in its proof.
The result states:

\begin{thm}\label{mnthm:locdual}
Let $(A,\m,\k)$ be a local noetherian DG-ring,
and let $R$ be a normalized dualizing DG-module over $A$.
Let $\bar{E} := E(\mrm{H}^0(A),\m) \in \opn{Mod}(\mrm{H}^0(A))$ be the injective $\mrm{H}^0(A)$-module which is the injective hull of the residue field $\k$.
Then for every $M \in \cat{D}^{+}_{\mrm{f}}(A)$, and any $n \in \mathbb{Z}$, there is an isomorphism
\[
\mrm{H}_{\m}^n(M) \cong \opn{Hom}_{\mrm{H}^0(A)}\left(\Ext^{-n}_A(M,R),\bar{E}\right)
\]
in $\opn{Mod}(\mrm{H}^0(A))$ which is functorial in $M$.
\end{thm}
This is repeated as Theorem \ref{thm:localdual} in the body of the paper.

\section{Preliminaries}

In this section we recall some basics about DG-rings that will be used throughout this paper.
A good reference for DG-rings and their derived categories is the book \cite{Ye}.
More preliminaries specifically about commutative DG-rings will be given in Section \ref{sec:prelcomm}.

By default all rings in this paper are associative and unital (but not commutative),
and all modules are left modules. Given a ring $A$, the category of left $A$-modules is denoted by $\opn{Mod}(A)$.

An associative differential graded-ring (abbreviated DG-ring) $A$ is a $\mathbb{Z}$-graded ring 
\[
A = \bigoplus_{n = -\infty}^{\infty} A^n
\]
equipped with a $\mathbb{Z}$-linear map $d:A \to A$ of degree $+1$ such that $d \circ d = 0$,
and such that
\[
d(a\cdot b) = d(a)\cdot b + (-1)^{i}\cdot a \cdot d(b)
\]
for all $a \in A^i$ and $b \in A^j$.

A DG-ring $A$ is called non-positive if $A^i = 0$ for all $i>0$.
In this case, $\mrm{H}^0(A)$ is a ring, and there is a canonical map of DG-rings $A \to \mrm{H}^0(A)$.

Associated to any DG-ring $A$ is its center DG-ring $Z(A)$,
The center $Z(A)$ is a DG subring of $A$ which is commutative.
See \cite[Definition 1.6]{YeSQ} for the precise definition.
If $A$ is a non-positive DG-ring then $Z(A)$ is a commutative non-positive DG-ring.

We denote by $\opn{DGMod}(A)$ the category of left DG-modules over $A$.
It is an abelian category. 
The category obtained from $\opn{DGMod}(A)$ by identifying homotopic morphisms is called the homotopy category of left DG-modules over $A$,
and is denoted by $\cat{K}(A)$.
The category obtained from $\opn{DGMod}(A)$ by formally inverting quasi-isomorphisms is called the derived category of left DG-modules over $A$,
and is denoted by $\cat{D}(A)$. 
Both $\cat{K}(A)$ and $\cat{D}(A)$ are are triangulated categories.
For any $M \in \cat{D}(A)$ and any $n \in \mathbb{Z}$,
cohomology defines a functor $\mrm{H}^n:\cat{D}(A) \to \opn{Mod}(\mrm{H}^0(A))$.
A left DG-module $M$ is called bounded above if $\mrm{H}^n(M) = 0$ for all $n>>0$,
bounded below if $\mrm{H}^n(M) = 0$ for all $n<<0$, and bounded if it both bounded above and bounded below.
The full triangulated subcategories consisting of bounded-above, bounded below and bounded DG-modules are denoted by $\cat{D}^{-}(A), \cat{D}^{+}(A)$ and $\cat{D}^{\mrm{b}}(A)$.
For $M \in \cat{D}(A)$, we set
\[
\inf(M) := \inf\{n \in \mathbb{Z} \mid \mrm{H}^n(M) \ne 0\},\quad \sup(M) := \sup\{n \in \mathbb{Z} \mid \mrm{H}^n(M) \ne 0\},
\]
and $\amp(M) := \sup(M) - \inf(M)$. We say that $\mrm{H}(M) \subseteq[a,b]$ if $a\le \inf(M)$ and $b \ge \sup(M)$.
A DG-module $M$ is called acyclic if $\mrm{H}^n(M) = 0$ for all $n \in \mathbb{Z}$.
A DG-module $M$ is called K-injective (respectively K-projective) if for every acyclic DG-module $X$,
the complex of abelian groups $\opn{Hom}_A(X,M)$ (respectively $\opn{Hom}_A(M,X)$) is acyclic.

A DG-module $I$ is called semi-injective if it is K-injective and the functor $\opn{Hom}_A(-,I)$ transforms injective maps to surjective maps.
Every DG-module has semi-injective and K-projective resolutions.

Given a non-positive DG-ring $A$ and $n \in \mathbb{Z}$,
there are two functors 
\[
\opn{smt}^{\le n}, \opn{smt}^{> n}: \cat{D}(A) \to \cat{D}(A),
\]
such that for any $M \in \cat{D}(A)$
\[
\mrm{H}^i(\opn{smt}^{\le n}(M)) =  \begin{cases}
                                          \mrm{H}^i(M) & \text{ if } i \le n \\
                                          0       & \text{ otherwise.}
                                         \end{cases}
\]
and
\[
\mrm{H}^i(\opn{smt}^{> n}(M)) =  \begin{cases}
                                          \mrm{H}^i(M) & \text{ if } i > n \\
                                          0       & \text{ otherwise.}
                                         \end{cases}
\]
Moreover, there are natural transformations
\[
\opn{smt}^{\le n} \to 1_{\cat{D}(A)}, \quad 1_{\cat{D}(A)} \to \opn{smt}^{> n}
\]
which induce a distinguished triangle
\[
\opn{smt}^{\le n}(M) \to M \to \opn{smt}^{> n}(M) \to \opn{smt}^{\le n}(M)[1]
\]
in $\cat{D}(A)$.

\section{Definitions of injective dimension over DG-rings}\label{sec:injdims}

The purpose of this section is to discuss the definition of the injective dimension of a left DG-module over a DG-ring.
Specifically, we would like to point out the following delicate point.
Following \cite[Section 2.I]{AF}, 
there are at least two sensible definitions for the injective dimension:

\begin{defx}\label{defA}
Let $A$ be a DG-ring, and let $M \in \cat{D}(A)$.
We define the injective dimension of $M$ to be the number
\[
\inf \{ n \in \mathbb{Z} \mid \Ext^i_A(N,M) = 0 \text{ for any } N \in \cat{D}^{\mrm{b}}(A) \text{ and any } i>n-\inf N\}
\]
\end{defx}

\begin{defx}\label{defB}
Let $A$ be a DG-ring, and let $M \in \cat{D}(A)$.
We define the injective dimension of $M$ to be the number
\[
\inf \{ n \in \mathbb{Z} \mid \Ext^i_A(N,M) = 0 \text{ for any } N \in \cat{D}^{+}(A) \text{ and any } i>n-\inf N\}
\]
\end{defx}

\noindent
In both definitions, we have set 
\[
\Ext^i_A(N,M) := \mrm{H}^i\left( \mrm{R}\opn{Hom}_A(N,M) \right).
\]

The difference between the two definitions is that in Definition \ref{defA} we test injective dimension on bounded DG-modules $N$,
while in Definition \ref{defB}, we test injective dimension on bounded below DG-modules.
Note that there is no need to test on DG-modules which are not bounded below, as in that case $\inf(N) = -\infty$, 
so the vanishing of $\Ext$ condition is always satisfied.

If $A$ is a ring, it is well known that both definitions coincide.
Most papers in the literature that impose injective dimension conditions over DG-rings use Definition \ref{defA} (see for example \cite{FIJ,FJ,Sh3}).
For a general DG-ring $A$, we do not know if the two definitions agree. However, assuming $A$ is non-positive, we have:

\begin{thm}\label{thm:buINJ}
Let $A$ be a non-positive DG-ring,
and let $M \in \cat{D}(A)$. 
Then the injective dimensions of $M$ as defined in Definition \ref{defA} and Definition \ref{defB} coincide.
\end{thm}
\begin{proof}
Denote the number
\[
\inf \{ n \in \mathbb{Z} \mid \Ext^i_A(N,M) = 0 \text{ for any } N \in \cat{D}^{\mrm{b}}(A) \text{ and any } i>n-\inf N\}
\]
by $m$. We may assume that $-\infty < m < \infty$, as otherwise there is nothing to prove.
Let $N \in \cat{D}^{+}(A)$. We must show that $\Ext^i_A(N,M) = 0$ for all $i>m-\inf N$.
For each $n \ge \inf(N)$, let $N_n := \opn{smt}^{\le n}(N)$.
Because $A$ is non-positive, $N_n$ is a DG-module, 
and is moreover a sub DG-module of $N$. 
We obtain a directed system of DG-modules
\[
N_n \subseteq N_{n+1} \subseteq N_{n+2} \subseteq \dots 
\]
such that $\varinjlim N_n = N$,
and such that $N_n \in \cat{D}^{\mrm{b}}(A)$.

Let $M \cong I$ be a semi-injective resolution. 
Then we have
\[
\mrm{R}\opn{Hom}_A(N,M) = \opn{Hom}_A(N,I) = \opn{Hom}_A(\varinjlim N_n,I) = \varprojlim \opn{Hom}_A(N_n,I).
\]
For each $n$, since the map $N_n \to N_{n+1}$ is injective (being an inclusion map),
semi-injectivity of $I$ implies that the map 
\[
\opn{Hom}_A(N_{n+1},I) \to \opn{Hom}_A(N_n,I)
\]
is surjective. 
Hence, the inverse system of complexes of abelian groups
\[
\left(\opn{Hom}_A(N_n,I)\right)_{n \in \mathbb{Z}}
\]
satisfies the Mittag-Leffler condition (in the sense of \cite[Definition 3.5.6]{We}).
It follows by \cite[Theorem 3.5.8]{We} that for each $i\in \mathbb{Z}$, 
there is a short exact sequence
\begin{equation}\label{eqn:ML-seq}
0 \to \varprojlim \phantom{ }^1 \mrm{H}^{i-1}\left(\opn{Hom}_A(N_n,I)\right) \to \mrm{H}^i \left(\opn{Hom}_A(N,I)\right) \to \varprojlim \mrm{H}^i\left(\opn{Hom}_A(N_n,I)\right) \to 0
\end{equation}
Since each $N_n \in \cat{D}^{\mrm{b}}(A)$ and moreover $\inf(N_n) = \inf(N)$,
by assumption we have that for each $n$,
\[
\mrm{H}^i\left(\opn{Hom}_A(N_n,I)\right) = 0
\]
for all $i>m - \inf(N_n) = m-\inf(N)$.
It now follows from (\ref{eqn:ML-seq}) that
\[
\mrm{H}^i \left(\opn{Hom}_A(N,I)\right) = 0
\]
for all $i> (m+1)-\inf(N)$.
It is thus remains to prove that
\[
\mrm{H}^i \left(\opn{Hom}_A(N,I)\right) = 0 
\]
for $i = (m+1)-\inf(N)$.

To show this,
let $N' := \opn{smt}^{\le \inf(N)}(N)$ and
$N'' := \opn{smt}^{> \inf(N)}(N)$.
Consider the distinguished triangle
\[
N' \to N \to N'' \to N'[1]
\]
in $\cat{D}(A)$. 
Applying the triangulated functor $\mrm{R}\opn{Hom}_A(-,M)$ we obtain a distinguished triangle
\[
\mrm{R}\opn{Hom}_A(N'',M) \to \mrm{R}\opn{Hom}_A(N,M) \to \mrm{R}\opn{Hom}_A(N',M) \to \mrm{R}\opn{Hom}_A(N'',M)[1]
\]
in $\cat{D}(\mathbb{Z})$.
Hence, for each $i$ there is an exact sequence
\begin{equation}\label{eqn:seq-of-ext}
\Ext^i_A(N'',M) \to \Ext^i_A(N,M) \to \Ext^i_A(N',M)
\end{equation}
Since $N'$ is a bounded DG-module with $\inf(N) = \inf(N')$, 
we have by assumption 
\[
\mrm{H}^{(m+1)-\inf(N)} \left(\mrm{R}\opn{Hom}_A(N',M)\right) = 0
\]
On the other hand, applying the first part of this proof to $N''$, we deduce that
\[
\mrm{H}^i\left(\mrm{R}\opn{Hom}_A(N'',M)\right) = 0
\]
for all $i> (m+1)-\inf(N'')$. 
But $\inf(N'') > \inf(N)$, so we deduce that
\[
\mrm{H}^{(m+1)-\inf(N)} \left(\mrm{R}\opn{Hom}_A(N'',M)\right) = 0
\]
Hence, we deduce from (\ref{eqn:seq-of-ext}) that
\[
\mrm{H}^{(m+1)-\inf(N)} \left(\mrm{R}\opn{Hom}_A(N,M)\right) = 0,
\]
and this completes the proof.
\end{proof}

\begin{notation}
In view of this result, we have one single definition of the injective dimension of a left DG-module over a non-positive DG-ring.
Given a non-positive DG-ring $A$, and $M \in \cat{D}(A)$,
we denote by $\injdim_A(M)$ the injective dimension of $M$ over $A$.
Thus, $\injdim_A(M)$ is the number (or $\pm \infty$)
\[
\inf \{ n \in \mathbb{Z} \mid \Ext^i_A(N,M) = 0 \text{ for any } N \in \cat{D}^{\mrm{b}}(A) \text{ and any } i>n-\inf N\},
\]
which is equal to the number
\[
\inf \{ n \in \mathbb{Z} \mid \Ext^i_A(N,M) = 0 \text{ for any } N \in \cat{D}^{+}(A) \text{ and any } i>n-\inf N\}.
\]
\end{notation}

We finish this section with the next result which is a useful tool for computing the injective dimension of a DG-module over a non-positive DG-ring.

\begin{thm}\label{thm:injDG}
Let $A$ be a non-positive DG-ring,
and let $M \in \cat{D}(A)$. 
Then there is an equality
\begin{equation*}
\injdim_A (M) = \injdim_{\mrm{H}^0(A)} \left(\RHom_A(\mrm{H}^0(A),M)\right)
\end{equation*}
\end{thm}
\begin{proof}
This was shown in \cite[Theorem 3.2]{Sh3}.
In that paper we assumed in addition that $A$ is commutative, 
but this assumption was not used in the proof of this result.
\end{proof}

\section{The category \texorpdfstring{$\INJ(A)$}{Inj(A)}}

In this section we begin investigating the main hero of this paper,
the category $\INJ(A)$.

\begin{dfn}
Let $A$ be a non-positive DG-ring.
We define the category $\INJ(A)$ to be the full subcategory of $\cat{D}(A)$ consisting of objects $M$ such that either $M \cong 0$ or
\begin{enumerate}
\item $\inf(M) = 0$, and
\item $\injdim_A(M) = 0$.
\end{enumerate}
\end{dfn}

\begin{exa}
If $A$ is a ring, then $M \in \INJ(A)$ if and only if there is an injective left $A$-module $I$ such that $M \cong I$ in $\cat{D}(A)$.
Thus, in this case, $\INJ(A)$ is equivalent to the full subcategory of $\opn{Mod}(A)$ consisting of injective left $A$-modules.
\end{exa}

For any non-positive DG-ring $A$, 
there is a DG-ring homomorphism $A \to \mrm{H}^0(A)$.
It induces a functor
\[
\mrm{R}\opn{Hom}_A(\mrm{H}^0(A),-) : \cat{D}(A) \to \cat{D}(\mrm{H}^0(A)),
\]
which will be extremely useful in the study of $\INJ(A)$.
One reason for its effectiveness is the following fact:

\begin{prop}\label{prop:RHomInf}
Let $A$ be a non-positive DG-ring.
Then for any $M \in \cat{D}^{+}(A)$,
we have that 
\[
\inf(\mrm{R}\opn{Hom}_A(\mrm{H}^0(A),M)) = \inf(M),
\]
and 
\[
\mrm{H}^{\inf(M)} (\mrm{R}\opn{Hom}_A(\mrm{H}^0(A),M)) \cong \mrm{H}^{\inf(M)}(M).
\]
\end{prop}
\begin{proof}
By shifting if necessary, we may assume without loss of generality that $\inf(M) = 0$.
Let $M \to I$ be a K-injective resolution such that $I^j = 0$ for $j<0$.
By definition,
\[
\mrm{R}\opn{Hom}_A(\mrm{H}^0(A),M) = \opn{Hom}_A(\mrm{H}^0(A),I)
\]
is concentrated in degrees $\ge 0$.
It is thus enough to calculate $\mrm{H}^0(\opn{Hom}_A(\mrm{H}^0(A),I))$.
Since $I$ is K-injective, we have 
\[
\mrm{H}^0(\opn{Hom}_A(\mrm{H}^0(A),I)) = \opn{Hom}_{\cat{K}(A)}(\mrm{H}^0(A),I).
\]
But $\mrm{H}^0(A)$ is concentrated in degree $0$, and $I$ in degrees $\ge 0$,
so we see that no two degree $0$ maps $f,g:\mrm{H}^0(A) \to I$ can be homotopic.
We deduce that
\[
\opn{Hom}_{\cat{K}(A)}(\mrm{H}^0(A),I) = \opn{Hom}_{\opn{DGMod}(A)}(\mrm{H}^0(A),I).
\]
The latter is by definition the collection of $A^0$-linear maps $f:\mrm{H}^0(A) \to I^0$ making the diagram
\[
\xymatrix{
0 \ar[r]\ar[d] & \mrm{H}^0(A) \ar[r]\ar[d]^f & 0 \ar[r]\ar[d] & 0 \ar[r]\ar[d] & \dots\\
0 \ar[r]       &     I^0      \ar[r]^{d^0}   & I^1 \ar[r]^{d^1}   & I^2 \ar[r] & \dots
}
\]
commutative. Commutativity of this diagram is equivalent to 
\[
\opn{Im}(f) \subseteq \ker(d^0) = \mrm{H}^0(I).
\]
Hence, we see that 
\[
\opn{Hom}_{\opn{DGMod}(A)}(\mrm{H}^0(A),I) = \opn{Hom}_{A^0}(\mrm{H}^0(A),\mrm{H}^0(I)) = \opn{Hom}_{\mrm{H}^0(A)}(\mrm{H}^0(A),\mrm{H}^0(I))
\]
which is equal to $\mrm{H}^0(I)$, as claimed.
\end{proof}

\begin{prop}\label{prop:RHomReflect}
Let $A$ be a non-positive DG-ring,
let $M, N \in \cat{D}^{+}(A)$,
and let $f:M \to N$ be a morphism in $\cat{D}(A)$.
If the morphism
\[
\mrm{R}\opn{Hom}_A(\mrm{H}^0(A),f) :\mrm{R}\opn{Hom}_A(\mrm{H}^0(A),M) \to \mrm{R}\opn{Hom}_A(\mrm{H}^0(A),N)
\]
in $\cat{D}(\mrm{H}^0(A))$ is an isomorphism, then $f$ is an isomorphism.
\end{prop}
\begin{proof}
Complete $f$ to a distinguished triangle
\[
M \xrightarrow{f} N \to K \to M[1]
\]
in $\cat{D}^{+}(A)$.
Applying the triangulated functor $\mrm{R}\opn{Hom}_A(\mrm{H}^0(A),-)$, we obtain a distinguished triangle
\begin{eqnarray}
\mrm{R}\opn{Hom}_A(\mrm{H}^0(A),M) \xrightarrow{\mrm{R}\opn{Hom}_A(\mrm{H}^0(A),f)} \mrm{R}\opn{Hom}_A(\mrm{H}^0(A),N) \to\nonumber\\
\mrm{R}\opn{Hom}_A(\mrm{H}^0(A),K) \to \mrm{R}\opn{Hom}_A(\mrm{H}^0(A),M)[1] \nonumber
\end{eqnarray}
in $\cat{D}(\mrm{H}^0(A))$.
Since we assumed that $\mrm{R}\opn{Hom}_A(\mrm{H}^0(A),f)$ is an isomorphism, we deduce that 
\[
\mrm{R}\opn{Hom}_A(\mrm{H}^0(A),K) \cong 0,
\]
so by Proposition \ref{prop:RHomInf} we deduce that $K \cong 0$, which implies that $f$ is an isomorphism because $K$ is the cone of $f$. 
\end{proof}

\begin{cor}\label{cor:HZGEN}
Let $A$ be a non-positive DG-ring.
Then $\mrm{H}^0(A)$ generates $\cat{D}^{+}(A)$.
\end{cor}
\begin{proof}
Given $M \in \cat{D}^{+}(A)$ such that $M \ncong 0$,
let $n = \inf(M)$. 
Then $\mrm{H}^n(M) \ne 0$.
By Proposition \ref{prop:RHomInf},
we have that
\[
\opn{Hom}_{\cat{D}(A)}(\mrm{H}^0(A),M[n]) = \mrm{H}^n(\mrm{R}\opn{Hom}_A(\mrm{H}^0(A),M)) \cong \mrm{H}^n(M) \ne 0,
\]
which proves the claim.
\end{proof}

Using Proposition \ref{prop:RHomInf}, we deduce the following characterization of $\INJ(A)$:
\begin{prop}\label{prop:INJCHAR}
Let $A$ be a non-positive DG-ring,
and let $M \in \cat{D}^{+}(A)$.
Then $M \in \INJ(A)$ if and only if
\[
\mrm{R}\opn{Hom}_A(\mrm{H}^0(A),M) \in \INJ(\mrm{H}^0(A)).
\]
In other words, $M \in \INJ(A)$ if and only if there is an injective left $\mrm{H}^0(A)$-module $\bar{J}$
and an isomorphism
\[
\mrm{R}\opn{Hom}_A(\mrm{H}^0(A),M) \cong \bar{J}
\]
in $\cat{D}(\mrm{H}^0(A))$.
\end{prop}
\begin{proof}
Suppose first $0 \ncong M \in \INJ(A)$.
Since the DG $A$-module $\mrm{H}^0(A)$ is concentrated in degree $0$,
the injective dimension condition on $M$ implies that
\[
\mrm{H}^i \left( \mrm{R}\opn{Hom}_A(\mrm{H}^0(A),M) \right) = 0
\]
for all $i > 0$. This implies that there is an $\mrm{H}^0(A)$-module $\bar{J}$
and an isomorphism
\[
\mrm{R}\opn{Hom}_A(\mrm{H}^0(A),M) \cong \bar{J}
\]
in $\cat{D}(\mrm{H}^0(A))$.
Moreover, by Theorem \ref{thm:injDG}, we have that
\[
\injdim_{\mrm{H}^0(A)}(\bar{J}) = \injdim_A(M) = 0,
\]
which implies that $\bar{J}$ is an injective module.

Conversely, let $0 \ncong M \in \cat{D}^{+}(A)$,
and assume there is an injective $\mrm{H}^0(A)$-module $\bar{J}$
and an isomorphism
\[
\mrm{R}\opn{Hom}_A(\mrm{H}^0(A),M) \cong \bar{J}
\]
in $\cat{D}(\mrm{H}^0(A))$.
By Proposition \ref{prop:RHomInf}, we deduce that $\inf(M) = 0$,
and by Theorem \ref{thm:injDG}, we see that $\injdim_A(M) = 0$.
Hence, $M \in \INJ(A)$.
\end{proof}

\begin{prop}\label{prop:HzOfINJ}
Let $A$ be a non-positive DG-ring,
and let $M \in \INJ(A)$.
Then the left $\mrm{H}^0(A)$-module $\mrm{H}^0(M)$ is injective.
\end{prop}
\begin{proof}
By Proposition \ref{prop:RHomInf}, we have that
\[
\mrm{H}^0(M) = \mrm{H}^0\left(\mrm{R}\opn{Hom}_A(\mrm{H}^0(A),M)\right),
\]
so the claim follows from Proposition \ref{prop:INJCHAR}.
\end{proof}

Propositions \ref{prop:INJCHAR} and \ref{prop:HzOfINJ} imply that both
$\mrm{R}\opn{Hom}_A(\mrm{H}^0(A),-)$ and $\mrm{H}^0(-)$ define functors $\INJ(A) \to \INJ(\mrm{H}^0(A))$.
We finish this section with the observation that these two functors coincide.

\begin{prop}\label{prop:RHomisHZ}
Let $A$ be a non-positive DG-ring.
Then the functors
\[
\mrm{R}\opn{Hom}_A(\mrm{H}^0(A),-) : \INJ(A) \to \INJ\left(\mrm{H}^0(A)\right)
\]
and 
\[
\mrm{H}^0(-) : \INJ(A) \to \INJ\left(\mrm{H}^0(A)\right)
\]
are naturally isomorphic.
\end{prop}
\begin{proof}
Given $M \in \INJ(A)$,
Proposition \ref{prop:INJCHAR} implies that the natural morphisms
\[
\mrm{R}\opn{Hom}_A(\mrm{H}^0(A),M) \to \opn{smt}^{\ge 0}\left(\mrm{R}\opn{Hom}_A(\mrm{H}^0(A),M)\right)
\]
and
\[
\opn{smt}^{\le 0}\left( \opn{smt}^{\ge 0}\left(\mrm{R}\opn{Hom}_A(\mrm{H}^0(A),M)\right) \right) \to \opn{smt}^{\ge 0}\left(\mrm{R}\opn{Hom}_A(\mrm{H}^0(A),M)\right)
\]
are isomorphisms.
Since
\[
\opn{smt}^{\le 0}\left( \opn{smt}^{\ge 0}\left(\mrm{R}\opn{Hom}_A(\mrm{H}^0(A),M)\right) \right) = \mrm{H}^0 \left( \mrm{R}\opn{Hom}_A(\mrm{H}^0(A),M) \right),
\]
the result follows from Proposition \ref{prop:RHomInf}.
\end{proof}

We will later show that these two isomorphic functors are natural equivalences.

\section{Contravariant functors of cohomological dimensional zero}

If $A$ is a ring, $I$ is an injective left $A$-module,
$M$ is a complex of left $A$-modules and $n \in \mathbb{Z}$,
then it is known that there is an isomorphism
\[
\mrm{H}^n\left(\opn{Hom}_A(M,I)\right) \cong \opn{Hom}_A\left(\mrm{H}^{-n}(M),I\right)
\]
which is functorial in $M,I$. See for example \cite[Corollary 2.12]{PSY1} for a proof of this classical fact.
In this section we will show that the same holds for $\INJ(A)$. 
To do this, it would be convenient to slightly generalize:

\begin{dfn}
Given two non-positive DG-rings $A,B$,
a contravariant triangulated functor $F:\cat{D}(A)\to \cat{D}(B)$
is said to have cohomological dimension $0$ 
if for any DG-module $M \in \cat{D}(A)$ such that $\mrm{H}(M)\subseteq [a,b]$,
where $-\infty \le a \le b \le \infty$,
we have that $\mrm{H}(F(M)) \subseteq [-b,-a]$.
\end{dfn}

This notion is related to $\INJ(A)$ via the following result:
\begin{prop}\label{prop:INJisZeroFunc}
Let $A$ be a non-positive DG-ring, 
and let $0 \ncong M \in \cat{D}^{+}(A)$.
Then $M \in \INJ(A)$ if and only if the functor
\[
F(-) := \mrm{R}\opn{Hom}_A(-,M) : \cat{D}(A) \to \cat{D}(\mathbb{Z})
\]
has cohomological dimension $0$.
\end{prop}
\begin{proof}
If $0 \ncong M \in \INJ(A)$, then $\inf(M) = 0$,
which implies that if $N \in \cat{D}^{-}(A)$ has $\sup(N) = b$,
then the definition of the $\opn{Hom}_A$ functor implies that
\[
\inf\left( \mrm{R}\opn{Hom}_A(N,M) \right) \ge -b.
\]
On the other hand, since $\injdim_A(M) = 0$,
we know that if $N \in \cat{D}^{+}(A)$ has $\inf(N) = a$,
then
\[
\sup\left( \mrm{R}\opn{Hom}_A(N,M) \right) \le -a.
\]
Hence, $F$ has cohomological dimension $0$.
Conversely, if $0 \ncong M \in \cat{D}^{+}(A)$ is a DG-module such that
$\mrm{R}\opn{Hom}_A(-,M)$ has cohomological dimension $0$,
then 
\[
\mrm{R}\opn{Hom}_A(\mrm{H}^0(A),M)
\]
is concentrated in degree $0$.
In particular, by Proposition \ref{prop:RHomInf}, we have that $\inf(M) = 0$.
Also, the cohomological dimension $0$ assumption clearly implies that $\injdim_A(M) = 0$.
Hence, $M \in \INJ(A)$, as claimed.
\end{proof}

For a non-positive DG-ring $A$, we denote by $\cat{D}^{\le 0}(A)$ the full subcategory of DG-modules $M$ with $\mrm{H}^i(M) = 0$ for all $i>0$.

\begin{lem}\label{lem:MtoHM}
Let $A$ be a non-positive DG-ring.
Then there is a natural morphism
\[
\alpha: 1_{\cat{D}^{\le 0}(A)} \to \mrm{H}^0(-)
\]
in $\cat{D}(A)$ such that for any $M \in \cat{D}^{\le 0}(A)$,
the map $\mrm{H}^0(\alpha_M)$ is an isomorphism.
Moreover, given $M \in \cat{D}^{\le 0}(A)$,
if one completes the map $\alpha_M$ into a distinguished triangle
\[
N \to M \xrightarrow{\alpha_M} \mrm{H}^0(M) \to N[1]
\]
in $\cat{D}(A)$, then $\sup(N) \le -1$.
\end{lem}
\begin{proof}
For any $M \in \cat{D}^{\le 0}(A)$,
let $P_M \cong M$ be a K-projective resolution such that $P_M^i = 0$ for all $i>0$.
Such exists since $\sup(M) \le 0$. 
K-projectivity of $P_M$ implies that this isomorphism is natural.
Let $\alpha_{P_M}$ be the natural surjection
\[
\xymatrixcolsep{2pc}
\xymatrixrowsep{2pc}
\xymatrix{
\dots \ar[r] &  P_M^{-2} \ar[r]\ar[d] & P_M^{-1} \ar[r]\ar[d] & P_M^{0} \ar[r]\ar[d] & 0 \ar[r]\ar[d] & \dots\\
\dots \ar[r] & 0      \ar[r]       & 0      \ar[r]       & \mrm{H}^0(P_M) \ar[r]& 0 \ar[r] & \dots
}
\]
and define $\alpha_M:M \to \mrm{H}^0(M)$ to be the composition 
\[
M \cong P_M \xrightarrow{\alpha_{P_M}} \mrm{H}^0(P_M) = \mrm{H}^0(M).
\]
It is clear from this definition that $\mrm{H}^0(\alpha_M)$ is an isomorphism.
Finally, letting 
\[
N \to M \xrightarrow{\alpha_M} \mrm{H}^0(M) \to N[1]
\]
be a distinguished triangle in $\cat{D}(A)$,
there is an exact sequence
\[
\mrm{H}^0(N) \to \mrm{H}^0(M) \xrightarrow{\mrm{H}^0(\alpha_M)} \mrm{H}^0(M) \to \mrm{H}^1(N) \to 0 \to 0 \to \mrm{H}^2(N) \to 0 \to \dots
\]
It is immediate from this sequence that $\mrm{H}^i(N) = 0$ for $i\ge 2$, and the fact that $\mrm{H}^0(\alpha_M)$ is an isomorphism implies that $\mrm{H}^0(N) = \mrm{H}^1(N) = 0$. Hence, $\sup(N)\le -1$, as claimed.
\end{proof}

\begin{thm}\label{thm:mainZero}
Let $A,B$ be non-positive DG-rings,
and let $F:\cat{D}(A) \to \cat{D}(B)$ be a contravariant triangulated functor of cohomological dimension $0$.
\begin{enumerate}
\item
There is an isomorphism 
\[
\Phi_F:\mrm{H}^0(F(-)) \to \mrm{H}^0\left(F(\mrm{H}^0(-))\right) 
\]
of functors $\cat{D}(A) \to \opn{Mod}(\mrm{H}^0(B))$.
\item If $G:\cat{D}(A) \to \cat{D}(B)$ is another contravariant triangulated functor of cohomological dimension $0$,
and $\Psi:F \to G$ is a morphism of triangulated functors,
for any $M \in \cat{D}(A)$ the diagram
\[
\xymatrixcolsep{4pc}
\xymatrixrowsep{3pc}
\xymatrix{
\mrm{H}^0(F(M)) \ar[r]^{\Phi_F(M)}\ar[d]_{\mrm{H}^0(\Psi_M)} & \mrm{H}^0\left(F(\mrm{H}^0(M))\right) \ar[d]^{\mrm{H}^0\left(\Psi_{\mrm{H}^0(M)}\right)} \\
\mrm{H}^0(G(M)) \ar[r]_{\Phi_G(M)}                           & \mrm{H}^0\left(G(\mrm{H}^0(M))\right) 
}
\]
is commutative.
\end{enumerate}
\end{thm}
\begin{proof}
We will prove these two statements simultaneously in three steps.

Step 1:
Let $M \in \cat{D}(A)$,
and set $M' = \opn{smt}^{\le 0}(M)$ and $M'' = \opn{smt}^{>0}(M)$.
Then there is a distinguished triangle 
\[
M' \to M \to M'' \to M'[1]
\]
in $\cat{D}(A)$, and hence a distinguished triangle 
\[
F(M'') \to F(M) \to F(M') \to F(M'')[1]
\]
in $\cat{D}(B)$. This triangle induces an exact sequence of $\mrm{H}^0(B)$-modules:
\begin{equation}\label{eqn:seq-fm}
\mrm{H}^0(F(M'')) \to \mrm{H}^0(F(M)) \to \mrm{H}^0(F(M')) \to \mrm{H}^1(F(M'')).
\end{equation}
Since $\inf(M'') \ge 1$, the fact that $F$ has cohomological dimension $0$ implies that 
\[
\sup(F(M'')) \le -1,
\]
so in particular 
\[
\mrm{H}^0(F(M'')) = \mrm{H}^1(F(M'')) = 0.
\]
Hence, the exact sequence (\ref{eqn:seq-fm}) implies that the map $\mrm{H}^0\left(F(M)\right) \to \mrm{H}^0\left(F(M')\right)$ is an isomorphism.
Since the morphism $M' = \opn{smt}^{\le 0}(M) \to M$ is natural in $M$, 
we deduce that the isomorphism 
\[
\beta^F_M: \mrm{H}^0\left(F(M)\right) \to \mrm{H}^0\left(F(\opn{smt}^{\le 0}(M))\right)
\]
is also natural in $M$. Note that we have denoted this natural isomorphism by $\beta^F_M$.
Naturality of $\Psi$ and the map $M' \to M$ imply that the diagram
\[
\xymatrix{
F(M) \ar[r]\ar[d]_{\Psi_M} & F(M')\ar[d]^{\Psi_{M'}}\\
G(M) \ar[r]       & G(M')
}
\]
is commutative, so applying the functor $\mrm{H}^0$ to it, we deduce that the diagram
\begin{equation}\label{eqn:bigd1}
\xymatrix{
\mrm{H}^0(F(M)) \ar[r]^{\beta^F_M}\ar[d]_{\mrm{H}^0(\Psi_M)} & \mrm{H}^0(F(M'))\ar[d]^{\mrm{H}^0(\Psi_{M'})}\\
\mrm{H}^0(G(M)) \ar[r]_{\beta^G_M}       & \mrm{H}^0(G(M'))
}
\end{equation}
is commutative.

Step 2: Let $M \in \cat{D}^{\le 0}(M)$,
and consider the distinguished triangle
\[
N \to M \xrightarrow{\alpha_M} \mrm{H}^0(M) \to N[1]
\]
in $\cat{D}(A)$ constructed in Lemma \ref{lem:MtoHM}.
Applying $F$, we obtain a distinguished triangle
\[
F\left(\mrm{H}^0(M)\right) \xrightarrow{F(\alpha_M)} F(M) \to F(N) \to F\left(\mrm{H}^0(M)\right)[1]
\]
in $\cat{D}(B)$. 
We obtain an exact sequence of $\mrm{H}^0(B)$-modules:
\begin{equation}\label{eqn:secondex}
\mrm{H}^{-1}\left(F(N)\right) \to \mrm{H}^0\left( F(\mrm{H}^0(M)) \right) \xrightarrow{\mrm{H}^0\left(F(\alpha_M)\right)} \mrm{H}^0\left(F(M)\right) \to \mrm{H}^0\left(F(N)\right).
\end{equation}
Since $\sup(N) \le -1$, the fact that $F$ has cohomological dimension $0$ implies that 
\[
\inf\left(F(N)\right) \ge 1.
\]
Hence, 
\[
\mrm{H}^{-1}\left(F(N)\right) = \mrm{H}^0\left(F(N)\right) = 0,
\]
so exactness of (\ref{eqn:secondex}) implies that 
\[
\mrm{H}^0\left( F(\mrm{H}^0(M)) \right) \xrightarrow{\mrm{H}^0\left(F(\alpha_M)\right)} \mrm{H}^0\left(F(M)\right)
\]
is an isomorphism. 
For such an $M \in \cat{D}^{\le 0}(M)$, let us set 
\[
\gamma^F_M := \mrm{H}^0\left(F(\alpha_M)\right).
\]
Since $\alpha_M$ was natural in $M$, we see that $\gamma^F$ is a natural isomorphism.
Applying the naturality of $\Psi$ to the map $\alpha_M$ implies that the diagram
\[
\xymatrixcolsep{4pc}
\xymatrix{
F(\mrm{H}^0(M)) \ar[r]^{F(\alpha_M)}\ar[d]_{\Psi_{\mrm{H}^0(M)}} & F(M)\ar[d]^{\Psi_M}\\
G(\mrm{H}^0(M)) \ar[r]_{G(\alpha_M)}       & G(M)
}
\]
is commutative, so applying $\mrm{H}^0$ to this diagram, we see that the diagram
\begin{equation}\label{eqn:bigd2}
\xymatrixcolsep{4pc}
\xymatrix{
\mrm{H}^0\left(F(\mrm{H}^0(M)\right) \ar[d]_{\mrm{H}^0\left(\Psi_{\mrm{H}^0(M)}\right)}  \ar[r]^{\gamma^F_M} & \mrm{H}^0(F(M)) \ar[d]^{\mrm{H}^0(\Psi_M)} \\
\mrm{H}^0\left(G(\mrm{H}^0(M)\right) \ar[r]_{\gamma^G_M}                                                     & \mrm{H}^0(G(M))
}
\end{equation}
is also commutative.

Step 3: Given $M \in \cat{D}(A)$,
let 
\[
\Phi_F(M) := (\gamma^F_{\opn{smt}^{\le 0}(M)})^{-1} \circ \beta^F_M : \mrm{H}^0(F(M)) \to \mrm{H}^0\left(F(\mrm{H}^0(\opn{smt}^{\le 0}(M))) \right)
\]
Being the composition of two natural isomorphisms, 
we see that $\Phi_F$ is a natural isomorphism.
Since by definition
\[
\mrm{H}^0\left(\opn{smt}^{\le 0}(M)\right) = \mrm{H}^0(M),
\]
we see that $\Phi_F$ is a natural isomorphism from $\mrm{H}^0(F(-))$ to $\mrm{H}^0(F(\mrm{H}^0(-)))$, proving the first claim of the theorem.
Combining the two commutative diagrams (\ref{eqn:bigd1}) and (\ref{eqn:bigd2}), 
we obtain a commutative diagram
\[
\xymatrixcolsep{4pc}
\xymatrixrowsep{3pc}
\xymatrix{
\mrm{H}^0\left(F(M)\right) \ar[r]^{\beta^F_M}\ar[d]_{\mrm{H}^0(\Psi_M)} & \mrm{H}^0\left(F(\opn{smt}^{\le 0}(M))\right) \ar[r]^{(\gamma^F_{\opn{smt}^{\le 0}(M)})^{-1}}\ar[d]^{\mrm{H}^0\left(\Psi_{\opn{smt}^{\le 0}(M)}\right)}   & \mrm{H}^0\left(F(\mrm{H}^0(M))\right) \ar[d]^{\mrm{H}^0\left(\Psi_{\mrm{H}^0(M)}\right)}\\
\mrm{H}^0\left(G(M)\right) \ar[r]_{\beta^G_M} & \mrm{H}^0\left(G(\opn{smt}^{\le 0}(M))\right) \ar[r]_{(\gamma^G_{\opn{smt}^{\le 0}(M)})^{-1}} & \mrm{H}^0\left(G(\mrm{H}^0(M))\right)
}
\]
where in the last column we have used the identification 
\[
\mrm{H}^0\left(M\right) = \mrm{H}^0\left(\opn{smt}^{\le 0}(M)\right).
\]
By the definition of $\Phi$, the composition of the two maps in the top row is equal to $\Phi_F(M)$,
and the composition of the two maps in the bottom row is equal to $\Phi_G(M)$, 
so commutativity of this diagram proves the second claim of the theorem.
\end{proof}

The converse of the above result is also true:
\begin{prop}\label{prop:convOfHZ}
Let $A,B$ be non-positive DG-rings,
and let $F:\cat{D}(A) \to \cat{D}(B)$ be a contravariant triangulated functor.
If there is an isomorphism 
\[
\mrm{H}^0(F(-)) \to \mrm{H}^0\left(F(\mrm{H}^0(-))\right) 
\]
of functors $\cat{D}(A) \to \opn{Mod}(\mrm{H}^0(B))$,
then $F$ is of cohomological dimension $0$.
\end{prop}
\begin{proof}
Let $M \in \cat{D}(A)$ such that $H(M) \subseteq [a,b]$.
For any $n \in \mathbb{Z}$,
we have by assumption:
\begin{eqnarray}
\mrm{H}^n(F(M)) = \mrm{H}^0(F(M)[n]) = \mrm{H}^0(F(M[-n])) \cong\nonumber\\
\mrm{H}^0(F(\mrm{H}^0(M[-n]))) = \mrm{H}^0(F(\mrm{H}^{-n}(M)).\nonumber
\end{eqnarray}
Hence, if $n \notin [-b,-a]$, so that $-n \notin [a,b]$,
we have that $\mrm{H}^{-n}(M) = 0$, so 
\[
\mrm{H}^n(F(M)) = 0,
\]
proving the claim.
\end{proof}

Given a DG-ring $A$, recall that we denote by $Z(A)$ the center of $A$.

\begin{thm}\label{thm:coho-of-rhom}
Let $A$ be a non-positive DG-ring, and let $n \in \mathbb{Z}$.
\begin{enumerate}
\item For any $M \in \cat{D}(A)$ and any $I \in \INJ(A)$ there is an isomorphism
\[
\mrm{H}^n \left(\mrm{R}\opn{Hom}_A(M,I)\right) \cong \opn{Hom}_{\mrm{H}^0(A)}(\mrm{H}^{-n}(M),\mrm{H}^0(I))
\]
in $\opn{Mod}\left(\mrm{H}^0(Z(A))\right)$ which is functorial in both $M$ and $I$.
\item For any $M \in \cat{D}(A\otimes_Z A^{\opn{op}})$ and any $I \in \INJ(A)$ there is an isomorphism
\[
\mrm{H}^n \left(\mrm{R}\opn{Hom}_A(M,I)\right) \cong \opn{Hom}_{\mrm{H}^0(A)}(\mrm{H}^{-n}(M),\mrm{H}^0(I))
\]
in $\opn{Mod}\left(\mrm{H}^0(A)\right)$ which is functorial in both $M$ and $I$.
\end{enumerate}
\end{thm}
\begin{proof}
For $M \in \cat{D}(A)$, let us consider $\mrm{R}\opn{Hom}_A(-,I)$ as a functor
$\cat{D}(A) \to \cat{D}(Z(A))$,
while for $M \in \cat{D}(A\otimes_Z A^{\opn{op}})$ we consider $\mrm{R}\opn{Hom}_A(-,I)$ as a functor
$\cat{D}(A\otimes_{\mathbb{Z}} A^{\opn{op}}) \to \cat{D}(A)$.

By Proposition \ref{prop:INJisZeroFunc}, 
the functor $\mrm{R}\opn{Hom}_A(-,I)$ has cohomological dimension $0$.
Since it is a contravariant triangulated functor, we have
\[
\mrm{H}^n \left(\mrm{R}\opn{Hom}_A(M,I)\right) = \mrm{H}^0\left(\mrm{R}\opn{Hom}_A(M[-n],I)\right).
\]
Hence, by Theorem \ref{thm:mainZero}(1), there is an isomorphism
\begin{eqnarray}\label{eqn:basic-func-hz}
 \mrm{H}^0\left(\mrm{R}\opn{Hom}_A(M[-n],I)\right) \cong\\
 \mrm{H}^0\left(\mrm{R}\opn{Hom}_A(\mrm{H}^0(M[-n]),I)\right) = \nonumber\\
\mrm{H}^0\left(\mrm{R}\opn{Hom}_A(\mrm{H}^{-n}(M),I)\right)\nonumber
\end{eqnarray}
which is functorial in $M$. 
By the theorem, for $M \in \cat{D}(A)$,
this isomorphism is $\mrm{H}^0(Z(A))$-linear,
while if $M \in \cat{D}(A\otimes_{\mathbb{Z}} A^{\opn{op}})$,
it is $\mrm{H}^0(A)$-linear.

If $I \to J$ is a morphism in $\INJ(A)$,
it induces a morphism of functors
\[
\mrm{R}\opn{Hom}_A(-,I) \to \mrm{R}\opn{Hom}_A(-,J),
\]
and hence by Theorem \ref{thm:mainZero}(2), the diagram
\[
\xymatrix{
\mrm{H}^0\left(\mrm{R}\opn{Hom}_A(M[-n],I)\right) \ar[r]^{\cong}\ar[d] & \mrm{H}^0\left(\mrm{R}\opn{Hom}_A(\mrm{H}^{-n}(M),I)\right)\ar[d]\\
\mrm{H}^0\left(\mrm{R}\opn{Hom}_A(M[-n],J)\right) \ar[r]^{\cong}       & \mrm{H}^0\left(\mrm{R}\opn{Hom}_A(\mrm{H}^{-n}(M),J)\right)
}
\]
is commutative. In other words, (\ref{eqn:basic-func-hz}) is functorial also in $I$.
By adjunction there is an isomorphism
\[
\mrm{H}^0\left(\mrm{R}\opn{Hom}_A(\mrm{H}^{-n}(M),I)\right) \cong
\mrm{H}^0\left(\mrm{R}\opn{Hom}_{\mrm{H}^0(A)}(\mrm{H}^{-n}(M),\mrm{R}\opn{Hom}_A(\mrm{H}^0(A),I)\right),
\]
functorial in $M,I$, and by Proposition \ref{prop:RHomisHZ},
\begin{eqnarray}
\mrm{H}^0\left(\mrm{R}\opn{Hom}_{\mrm{H}^0(A)}(\mrm{H}^{-n}(M),\mrm{R}\opn{Hom}_A(\mrm{H}^0(A),I))\right) \cong \nonumber\\
\mrm{H}^0\left(\mrm{R}\opn{Hom}_{\mrm{H}^0(A)}(\mrm{H}^{-n}(M),\mrm{H}^0(I))\right) =\nonumber\\
\opn{Hom}_{\mrm{H}^0(A)}(\mrm{H}^{-n}(M),\mrm{H}^0(I)) \nonumber
\end{eqnarray}
proving the claim.
\end{proof}

As an immediate corollary of this result, we may compute the cohomologies of elements of $\INJ(A)$:

\begin{cor}\label{cor:coho-of-INJ}
Let $A$ be a non-positive DG-ring, and let $I \in \INJ(A)$.
Then for any $n \in \mathbb{N}$, there is an isomorphism
\[
\mrm{H}^n(I) \cong \opn{Hom}_{\mrm{H}^0(A)}\left(\mrm{H}^{-n}(A),\mrm{H}^0(I)\right).
\]
in $\opn{Mod}(\mrm{H}^0(A))$ which is functorial in $I$.
\end{cor}
\begin{proof}
Apply Theorem \ref{thm:coho-of-rhom}(2) to $M = A \in \cat{D}(A\otimes_{\mathbb{Z}} A^{\opn{op}})$.
\end{proof}

We finish this section with the following categorical characterization of $\INJ(A)$:

\begin{cor}\label{cor:catCharI}
Let $A$ be a non-positive DG-ring,
and let $I \in \cat{D}(A)$.
Then the following are equivalent:
\begin{enumerate}
\item $I \in \INJ(A)$.
\item For any morphism $f:M \to N$ in $\cat{D}(A)$ such that $\mrm{H}^0(f):\mrm{H}^0(M) \to \mrm{H}^0(N)$ is an injective map,
the map 
\[
\opn{Hom}_{\cat{D}(A)}(f,I): \opn{Hom}_{\cat{D}(A)}(N,I) \to \opn{Hom}_{\cat{D}(A)}(M,I)
\]
is surjective.
\end{enumerate}
\end{cor}
\begin{proof}
(1) $\implies$ (2):
Let $f:M \to N$ be a morphism in $\cat{D}(A)$ such that $\mrm{H}^0(f)$ is an injective map.
By Theorem \ref{thm:coho-of-rhom} there is a commutative diagram
\[
\xymatrix{
\mrm{H}^0\left(\mrm{R}\opn{Hom}_A(N,I)\right) \ar[r]\ar[d] & \opn{Hom}_{\mrm{H}^0(A)}\left(\mrm{H}^0(N),\mrm{H}^0(I)\right)\ar[d]\\
\mrm{H}^0\left(\mrm{R}\opn{Hom}_A(M,I)\right) \ar[r] & \opn{Hom}_{\mrm{H}^0(A)}\left(\mrm{H}^0(M),\mrm{H}^0(I)\right)
}
\]
such that the horizontal maps are isomorphisms.
The leftmost vertical map is exactly 
\[
\opn{Hom}_{\cat{D}(A)}(f,I): \opn{Hom}_{\cat{D}(A)}(N,I) \to \opn{Hom}_{\cat{D}(A)}(M,I)
\]
so it is enough to show that 
\[
\opn{Hom}_{\mrm{H}^0(A)}\left(\mrm{H}^0(N),\mrm{H}^0(I)\right) \to \opn{Hom}_{\mrm{H}^0(A)}\left(\mrm{H}^0(M),\mrm{H}^0(I)\right)
\]
is surjective, and this is true because by Proposition \ref{prop:HzOfINJ}, the $\mrm{H}^0(A)$-module $\mrm{H}^0(I)$ is an injective module.

(2) $\implies$ (1): Let $I \in \cat{D}(A)$ be a DG-module that satisfies (2).
For any $n > 0$, non-positivity of $A$ implies that $\mrm{H}^0(A[n]) = 0$.
Hence, the zero map $f:A[n] \to 0$ satisfies that $\mrm{H}^0(f)$ is injective.
Hence, by assumption, the map
\[
0 = \opn{Hom}_{\cat{D}(A)}(0,I) \to \opn{Hom}_{\cat{D}(A)}(A[n],I) = \mrm{H}^{-n}(I)
\]
is surjective, so that $\inf(I) \ge 0$. In particular, $I \in \cat{D}^{+}(A)$.
Given $N \in \cat{D}^{+}(A)$,
let $i < \inf(N)$. Then $\mrm{H}^0(N[i]) = 0$, so the zero map 
$g:N[i] \to 0$ satisfies that $\mrm{H}^0(g)$ is injective,
and hence by assumption, 
\[
0 = \opn{Hom}_{\cat{D}(A)}(0,I) \to \opn{Hom}_{\cat{D}(A)}(N[i],I)
\]
is surjective, so that
\begin{equation}\label{eqn:INJdimZHC}
\mrm{H}^{-i}\left(\mrm{R}\opn{Hom}_A(N,I)\right) = 0.
\end{equation}
Applying this to $N = \mrm{H}^0(A)$ and using the fact that $\inf(I) \ge 0$ implies by Proposition \ref{prop:RHomInf} that
the cohomology of $\mrm{R}\opn{Hom}_A(\mrm{H}^0(A),I)$ is concentrated in degree $0$.
Since $I \in \cat{D}^{+}(A)$,
if $\mrm{H}^0(I) = 0$ then by Proposition \ref{prop:RHomInf} 
\[
\mrm{H}^0\left(\mrm{R}\opn{Hom}_A(\mrm{H}^0(A),I)\right) = 0
\]
so that $\mrm{R}\opn{Hom}_A(\mrm{H}^0(A),I) = 0$, and hence by Corollary \ref{cor:HZGEN}, $I \cong 0$, so that $I \in \INJ(A)$.
We may hence assume that $\mrm{H}^0(I) \ne 0$, and then $\inf(I) = 0$,
and (\ref{eqn:INJdimZHC}) implies that $\injdim_A(I) = 0$, so again $I \in \INJ(A)$.
\end{proof}

\begin{rem}
Let $A$ be a non-positive DG-ring,
and let $\mathcal{I}$ be the collection of morphisms $f:M \to N$ in $\cat{D}(A)$ such that $\mrm{H}^0(f)$ is an injective map.
Then the above corollary may be stated as: the elements of $\INJ(A)$ are exactly the $\mathcal{I}$-injective objects of $\cat{D}(A)$.
\end{rem}

\section{An equivalence of categories of injectives}\label{sec:eqv}
 
In this section we will show that for a non-positive DG-ring $A$,
the functor 
\[
\mrm{H}^0(-):\INJ(A) \to \INJ(\mrm{H}^0(A))
\]
is an equivalence of categories.

\begin{prop}\label{prop:HZff}
Let $A$ be a non-positive DG-ring.
Then the functor
\[
\mrm{H}^0(-) : \INJ(A) \to \INJ(\mrm{H}^0(A))
\]
is fully faithful.
\end{prop}
\begin{proof}
Given $I,J \in \INJ(A)$, we have by definition:
\[
\opn{Hom}_{\INJ(A)}(I,J) = \opn{Hom}_{\cat{D}(A)}(I,J) =  \mrm{H}^0\left(\mrm{R}\opn{Hom}_A(I,J)\right).
\]
By Theorem \ref{thm:coho-of-rhom}, there is an isomorphism
\[
\mrm{H}^0\left(\mrm{R}\opn{Hom}_A(I,J)\right) \cong \opn{Hom}_{\mrm{H}^0(A)}\left(\mrm{H}^0(I),\mrm{H}^0(J)\right) = \opn{Hom}_{\INJ\left(\mrm{H}^0(A)\right)}\left(\mrm{H}^0(I),\mrm{H}^0(J)\right),
\]
and by the proof of that theorem, this isomorphism is induced by $\mrm{H}^0(-)$, proving the claim.
\end{proof}

\begin{prop}\label{prop:summands}
Let $A$ be a non-positive DG-ring.
Then $\INJ(A)$ is closed under direct summands.
\end{prop}
\begin{proof}
Suppose $J = I_1 \oplus I_2$, and that $J \in \INJ(A)$.
For each $n \in \mathbb{Z}$,
this implies that $\mrm{H}^n(J) = \mrm{H}^n(I_1) \oplus \mrm{H}^n(I_2)$,
so that $I_1 \in \cat{D}^{+}(A)$.
Moreover, for each $n \in \mathbb{Z}$,
\[
\Ext^n_A(\mrm{H}^0(A),I) = \Ext^n_A(\mrm{H}^0(A),I_1) \oplus \Ext^n_A(\mrm{H}^0(A),I_2),
\]
so by Proposition \ref{prop:INJCHAR} one has $\Ext^n_A(\mrm{H}^0(A),I_1) = 0$ for all $n \ne 0$,
and $\Ext^0_A(\mrm{H}^0(A),I_1)$ is a direct summand of an injective $\mrm{H}^0(A)$-module, so it is also an injective $\mrm{H}^0(A)$-module.
Hence, by Proposition \ref{prop:INJCHAR}, $I_1 \in \INJ(A)$.
\end{proof}

Recall that if $\mathcal{C}$ is a category,
a morphism $e:M \to M$ in $\mathcal{C}$ is an idempotent if $e\circ e = e$.
An idempotent $e$ is called a split idempotent if
there is an object $N \in \mathcal{C}$
and $\mathcal{C}$-morphisms $r:M \to N$ and $s:N \to M$ such that $s \circ r = e$ and $r \circ s = 1_N$.
In this case one says that $r,s$ splits $e$. 
Moreover, in this case the morphism $s$ is a monomorphism, 
and one says that $s$ is a split monomorphism.

The following is well known:
\begin{prop}\label{prop:uniqsplit}
Let $\mathcal{C}$ be a category,
and let $e:M \to M$ be an idempotent morphism. 
Then any two splittings of $e$ are isomorphic.
Explicitly, if $r:M \to N$, $s:N \to M$ and $r':M \to N'$, $s':N' \to M$ both split $e$ then $N$ and $N'$ are isomorphic.
\end{prop}
\begin{proof}
Let $f = r'\circ s :N \to N'$ and $g = r\circ s': N' \to N$.
Then 
\[
f \circ g = r' \circ s \circ  r \circ s' = r' \circ e \circ s' = r' \circ s' \circ r' \circ s' = 1_{N'} \circ 1_{N'} = 1_{N'},
\]
and similarly
\[
g \circ f = r \circ s' \circ r' \circ s = r \circ e \circ s = r \circ s \circ r \circ s = 1_N \circ 1_N  = 1_N,
\]
proving the claim.
\end{proof}

\begin{lem}\label{lem:lift-inclusion}
Let $A$ be non-positive DG-ring,
and let $\bar{I}$ be an injective $\mrm{H}^0(A)$-module.
Suppose that there is an injective $\mrm{H}^0(A)$-module $\bar{J}$,
such that $\bar{I} \subseteq \bar{J}$,
and such that there is some $J \in \INJ(A)$ with $\mrm{H}^0(J) = \bar{J}$.
Then there exists $I \in \INJ(A)$ such that $\mrm{H}^0(I) \cong \bar{I}$.
\end{lem}
\begin{proof}
Since $\bar{I}$ is an injective module,
the inclusion map $\bar{I} \inj \bar{J}$ must split.
Hence, there is an $\mrm{H}^0(A)$-linear idempotent map $\bar{e}:\bar{J} \to \bar{J}$
and $\mrm{H}^0(A)$-linear maps $\bar{r}:\bar{J} \to \bar{I}$, $\bar{s}:\bar{I} \to \bar{J}$ such that
$\bar{s} \circ \bar{r} = \bar{e}$ and $\bar{r} \circ \bar{s} = 1_{\bar{I}}$.

By Proposition \ref{prop:HZff},
there is a morphism $e:J \to J$ in $\cat{D}(A)$ such that $\mrm{H}^0(e) = \bar{e}$.
Moreover, the fact that $\bar{e}$ is idempotent and that $\mrm{H}^0$ is fully faithful implies that $e$ is also idempotent.
Hence, by \cite[Proposition 3.2]{BN}, $e$ splits.
Explicitly, there is some $I \in \cat{D}(A)$ and maps $r:J \to I$ and $s:I \to J$ such that
\begin{equation}\label{eqn:esplits}
s \circ r = e, \quad r\circ s = 1_I.
\end{equation}
In particular, $I$ is a direct summand of $J$, so by Proposition \ref{prop:summands},
$I \in \INJ(A)$.
Applying the functor $\mrm{H}^0(-)$ to (\ref{eqn:esplits}), we see that
\[
\mrm{H}^0(s) \circ \mrm{H}^0(r) = \bar{e}, \quad \mrm{H}^0(r)\circ \mrm{H}^0(s) = 1_{\mrm{H}^0(I)}.
\]
Hence, $(\mrm{H}^0(r),\mrm{H}^0(s))$ and $(\bar{r},\bar{s})$ are both splittings of $\bar{e}$,
so by Proposition \ref{prop:uniqsplit}, there is an $\mrm{H}^0(A)$-linear isomorphism $\mrm{H}^0(I) \cong \bar{I}$,
proving the claim.
\end{proof}

If $A \to B$ is a ring map and $I$ is an injective $A$-module then $\opn{Hom}_A(B,I)$ is an injective $B$-module.
More generally:

\begin{prop}\label{prop:INJFunc}
Let $A \to B$ be a map of non-positive DG-rings,
and let $I \in \INJ(A)$.
Then
\[
\mrm{R}\opn{Hom}_A(B,I) \in \INJ(B).
\]
\end{prop}
\begin{proof}
Since $\sup(B) = 0$ and $I \in \INJ(A)$,
we know that $\mrm{R}\opn{Hom}_A(B,I) \in \cat{D}^{+}(B)$.
Moreover, by adjunction
\[
\mrm{R}\opn{Hom}_B\left(\mrm{H}^0(B),\mrm{R}\opn{Hom}_A(B,I)\right) \cong \mrm{R}\opn{Hom}_{\mrm{H}^0(A)}\left(\mrm{H}^0(B),\mrm{R}\opn{Hom}_A(\mrm{H}^0(A),I)\right).
\]
By Proposition \ref{prop:INJCHAR},
we know that $\mrm{R}\opn{Hom}_A(\mrm{H}^0(A),I)$ is isomorphic to an injective $\mrm{H}^0(A)$-module,
so by the above isomorphism
\[
\mrm{R}\opn{Hom}_B\left(\mrm{H}^0(B),\mrm{R}\opn{Hom}_A(B,I)\right)
\]
is an injective $\mrm{H}^0(B)$-module.
Hence, by Proposition \ref{prop:INJCHAR}, $\mrm{R}\opn{Hom}_A(B,I) \in \INJ(B)$.

\end{proof}

Here is the main result of this section.

\begin{thm}\label{thm:eqv}
Let $A$ be a non-positive DG-ring.
Then the functor
\[
\mrm{H}^0(-) : \INJ(A) \to \INJ(\mrm{H}^0(A))
\]
is an equivalence of categories.
\end{thm}
\begin{proof}
By Proposition \ref{prop:HZff} the functor 
\[
\mrm{H}^0(-) : \INJ(A) \to \INJ(\mrm{H}^0(A))
\]
is fully faithful, so it remains to show that it is essentially surjective.
Let us give two proofs of this important fact.
First proof:
Let $\bar{I}$ be an injective $\mrm{H}^0(A)$-module.
Since $\opn{Hom}_Z(\mrm{H}^0(A),\mathbb{Q}/\mathbb{Z})$ is an injective cogenerator of $\opn{Mod}(\mrm{H}^0(A))$,
there is some cardinal number $\mu$ and an injective $\mrm{H}^0(A)$-linear map 
\begin{equation}\label{eqn:injImbed}
\bar{I} \inj \prod_{\mu} \left(\opn{Hom}_Z(\mrm{H}^0(A),\mathbb{Q}/\mathbb{Z})\right) = 
\opn{Hom}_Z\left(\mrm{H}^0(A),\left(\prod_{\mu} \mathbb{Q}/\mathbb{Z}\right)\right)
\end{equation}
Let 
\[
J := \opn{Hom}_Z\left(A,\left(\prod_{\mu} \mathbb{Q}/\mathbb{Z}\right)\right) \in \cat{D}(A).
\]
By Proposition \ref{prop:INJFunc},
$J \in \INJ(A)$.
Hence, by Proposition \ref{prop:RHomInf} there is an isomorphism
\begin{equation}\label{eqn:hzofJ}
\mrm{H}^0(J) \cong \mrm{H}^0\left(\mrm{R}\opn{Hom}_A(\mrm{H}^0(A),J)\right)
\end{equation}
By definition of $J$ and adjunction, we have
\begin{eqnarray}
\mrm{R}\opn{Hom}_A(\mrm{H}^0(A),J) = 
\mrm{R}\opn{Hom}_A\left(\mrm{H}^0(A),\opn{Hom}_{\mathbb{Z}}\left(A,\left(\prod_{\mu} \mathbb{Q}/\mathbb{Z}\right)\right)\right) \cong\nonumber\\
\opn{Hom}_Z\left(\mrm{H}^0(A),\left(\prod_{\mu} \mathbb{Q}/\mathbb{Z}\right)\right)\nonumber.
\end{eqnarray}
As the latter is an $\mrm{H}^0(A)$-module which is concentrated in degree zero, we deduce from (\ref{eqn:hzofJ}) that
\[
\mrm{H}^0(J) \cong \opn{Hom}_Z\left(\mrm{H}^0(A),\left(\prod_{\mu} \mathbb{Q}/\mathbb{Z}\right)\right).
\]
It follows by Lemma \ref{lem:lift-inclusion} and (\ref{eqn:injImbed})
that there exists $I \in \INJ(A)$ such that $\mrm{H}^0(I) \cong \bar{I}$.
Hence,
\[
\mrm{H}^0(-) : \INJ(A) \to \INJ(\mrm{H}^0(A))
\]
is essentially surjective, so it is an equivalence of categories.

Second proof:
Let $0 \ne \bar{I}$ be an injective $\mrm{H}^0(A)$-module,
and let $F:\cat{D}(A) \to \opn{Mod}(\mathbb{Z})$ be the functor
\[
F(M) := \opn{Hom}_{\mrm{H}^0(A)}(\mrm{H}^0(M),\bar{I}).
\]
Then $F$ is a homological functor.
Hence, by Neeman's Brown representability theorem (\cite[Theorem 3.1]{Ne}),
$F$ is representable. Thus, there is some $I \in \cat{D}(A)$ and an isomorphism
\[
\opn{Hom}_{\cat{D}(A)}(-,I) \cong F(-)
\]
of functors $\cat{D}(A) \to \opn{Mod}(\mathbb{Z})$.
In other words, we have an isomorphism
\[
\mrm{H}^0\left(\mrm{R}\opn{Hom}_A(M,I)\right) \cong \opn{Hom}_{\mrm{H}^0(A)}(\mrm{H}^0(M),\bar{I})
\]
for all $M \in \cat{D}(A)$. Using this isomorphism for $M = A[n]$ with $n \ge 0$, we see that $\mrm{H}^{-n}(I) = 0$, 
for $n > 0$, while $\mrm{H}^0(I) \ne 0$, so that $\inf(I) = 0$. 
On the other hand, this isomorphism clearly implies that $\injdim_A(I) = 0$,
so that $I \in \INJ(A)$. Combining the above isomorphism with Theorem \ref{thm:coho-of-rhom}, 
and using the fact that $\mrm{H}^0:\cat{D}(A)\to \opn{Mod}(\mrm{H}^0(A))$ is essentially surjective
we deduce that the functors $\opn{Mod}(\mrm{H}^0(A)) \to \opn{Mod}(\mathbb{Z})$ given by
$\opn{Hom}_{\mrm{H}^0(A)}(-,\bar{I})$ and $\opn{Hom}_{\mrm{H}^0(A)}(-,\mrm{H}^0(I))$ are isomorphic.
Hence, by the Yoneda lemma, $\mrm{H}^0(I) \cong \bar{I}$, proving the claim.
\end{proof}

\begin{rem}
In \cite[Corollary 7.2.2.19]{Lu2}, 
Lurie proved a result dual to the above theorem about projective modules over $E_1$ rings.
\end{rem}

\begin{cor}
Let $A$ be a non-positive DG-ring,
and let $I \in \INJ(A)$.
Then $I$ is an indecomposable object of $\cat{D}(A)$
if and only if its endomorphism ring
\[
\opn{Hom}_{\cat{D}(A)}(I,I)
\]
is a local ring.
\end{cor}
\begin{proof}
In any additive category an object whose endomorphism ring is local is indecomposable.
Conversely, suppose $I$ is an indecomposable object of $\cat{D}(A)$.
Then it is also an indecomposable object of $\INJ(A)$.
By Theorem \ref{thm:eqv}, $\mrm{H}^0(I)$ is an indecomposable object of $\INJ(\mrm{H}^0(A))$.
Thus, $\mrm{H}^0(I)$ is an indecomposable injective $\mrm{H}^0(A)$-module,
so by \cite[Proposition 2.6]{Ma} its endomorphism ring is a local ring.
Hence, the equality
\[
\opn{Hom}_{\cat{D}(A)}(I,I) = \opn{Hom}_{\INJ(A)}(I,I) = \opn{Hom}_{\mrm{H}^0(A)}(\mrm{H}^0(I),\mrm{H}^0(I))
\]
implies that $\opn{Hom}_{\cat{D}(A)}(I,I)$ is also a local ring.
\end{proof}

Here is another categorical characterization of $\INJ(A)$:
\begin{prop}\label{prop:catCharII}
Let $A$ be a non-positive DG-ring, 
and let $I \in \cat{D}(A)$.
Then $I \in \INJ(A)$ if and only if for any morphism
$f:I \to M$ in $\cat{D}(A)$ such that $\mrm{H}^0(f): \mrm{H}^0(I) \to \mrm{H}^0(M)$ is an injective map,
the morphism $f$ is a split monomorphism.
\end{prop}
\begin{proof}
Suppose $I \in \INJ(A)$,
and let $f:I \to M$ be a morphism in $\cat{D}(A)$ such that $\mrm{H}^0(f): \mrm{H}^0(I) \to \mrm{H}^0(M)$ is an injective map.
By Proposition \ref{prop:HzOfINJ}, the $\mrm{H}^0(A)$-module $\mrm{H}^0(I)$ is injective.
Hence, the injective map
$\mrm{H}^0(f): \mrm{H}^0(I) \to \mrm{H}^0(M)$ is a split monomorphism,
so there is a map $\bar{g}: \mrm{H}^0(M) \to \mrm{H}^0(I)$ such that
\[
\bar{g} \circ \mrm{H}^0(f) = 1_{\mrm{H}^0(I)}.
\]
It follows by Theorem \ref{thm:coho-of-rhom} that there is a map $g:M \to I$ in $\cat{D}(A)$
such that $\mrm{H}^0(g) = \bar{g}$.
Hence, 
\[
\mrm{H}^0(g\circ f) = \mrm{H}^0(g) \circ \mrm{H}^0(f) = \bar{g} \circ \mrm{H}^0(f) = 1_{\mrm{H}^0(I)}.
\]
Hence, by Proposition \ref{prop:HZff} we deduce that $g\circ f = 1_I$,
so that $f$ is a split monomorphism.

Conversely, let $I \in \cat{D}(A)$,
and suppose that for any morphism
$f:I \to M$ in $\cat{D}(A)$ such that $\mrm{H}^0(f): \mrm{H}^0(I) \to \mrm{H}^0(M)$ is an injective map,
the morphism $f$ is a split monomorphism.

Let $\bar{J}$ be an injective $\mrm{H}^0(A)$-module such that there is an embedding
$\bar{f}: \mrm{H}^0(I) \inj \bar{J}$. 
By Theorem \ref{thm:eqv}, there exists $J \in \INJ(A)$ such that $\mrm{H}^0(J) = \bar{J}$.
Since $J \in \INJ(A)$ and
\[
\bar{f} \in \opn{Hom}_{\mrm{H}^0(A)}(\mrm{H}^0(I),\mrm{H}^0(J)),
\]
we deduce by Theorem \ref{thm:coho-of-rhom} that there exists some
\[
f \in \opn{Hom}_{\cat{D}(A)}(I,J)
\]
such that $\mrm{H}^0(f) = \bar{f}$. But $\bar{f}$ is an injective map, 
so by our assumption on $I$ we deduce that $f$ is a split monomorphism.
Hence, $I$ is a direct summand of $J$, 
so by Proposition \ref{prop:summands} we deduce that $I \in \INJ(A)$.
\end{proof}

\begin{rem}
The paper \cite{GP} makes a detailed study of injective objects in a triangulated category.
The definition of an injective object given there (\cite[Definition 3.1]{GP}) is an object such that any injective map (in a suitable sense) into it splits.
This is similar to the above proposition.
However, the family of injective maps in their setup has the property that it is closed under shifts.
Hence, the results of this paper do not fall under their framework.
\end{rem}

We are now ready to prove Theorem \ref{thmI} from the introduction:
\begin{proof}
(1) $\iff (2)$ follows from Theorem \ref{thm:buINJ}.
(2) $\iff (3)$ follows from Proposition \ref{prop:INJCHAR} and Proposition \ref{prop:INJFunc}.
(2) $\iff (4)$ follows from Theorem \ref{thm:coho-of-rhom} and Proposition \ref{prop:convOfHZ}.
(2) $\iff (5)$ is Corollary \ref{cor:catCharI},
and (2) $\iff$ (6) is Proposition \ref{prop:catCharII}.
\end{proof}

\section{The derived Bass-Papp Theorem}

Given a ring $A$, recall that the Bass-Papp Theorem says that $A$ is left noetherian 
if and only if the category of left injective $A$-modules is closed under arbitrary direct sums.
It was proved by Papp in \cite[Theorem 1]{Pa} and independently by Bass in \cite[Theorem 1.1]{Ba}.
The purpose of this section is to generalize this theorem to the category of non-positive DG-rings.

Thus, given a non-positive DG-ring $A$, we ask: when is $\INJ(A)$ closed under arbitrary direct sums?
At first glance it might seem that Theorem \ref{thm:eqv} already gives an answer, but this is not the case.
The equivalence of categories between $\INJ(A)$ and $\INJ(\mrm{H}^0(A))$ will only give us information about 
coproducts relative to $\INJ(A)$, and these might be different from the coproducts in $\opn{DGMod}(A)$.
Indeed, the main result of this section shows that it can happen that $\mrm{H}^0(A)$ is left noetherian (so that $\INJ(\mrm{H}^0(A))$ is closed under arbitrary direct sums), but $\INJ(A)$ is not closed under arbitrary direct sums.

Given a DG-ring $A$, a DG-module $M$
and a collection $\{N_{\alpha}\}_{\alpha \in J}$ of DG-modules over $A$,
the canonical inclusions $N_{\alpha} \inj \bigoplus_{\alpha \in J} N_{\alpha}$ induce for each $\alpha \in J$ maps
\[
\varphi_{\alpha}:\opn{Hom}_A(M,N_{\alpha}) \to \opn{Hom}_A\left(M,\bigoplus_{\alpha \in J} N_{\alpha}\right).
\]
By the universal property of direct sums, 
there is a unique $A$-linear map
\begin{equation}\label{eqn:compact}
\varphi:\bigoplus_{\alpha \in J} \opn{Hom}_A(M,N_{\alpha}) \to \opn{Hom}_A\left(M,\bigoplus_{\alpha \in J} N_{\alpha}\right)
\end{equation}
such that for each $\alpha \in J$, the diagram
\[
\xymatrix{
\bigoplus_{\alpha \in J} \opn{Hom}_A(M,N_{\alpha}) \ar[r]^{\varphi} & \opn{Hom}_A\left(M,\bigoplus_{\alpha \in J} N_{\alpha}\right)\\
\opn{Hom}_A(M,N_{\alpha})\ar@{^{(}->}[u] \ar[ur]_{\varphi_{\alpha}}
}
\]
is commutative. 
It is clear that (\ref{eqn:compact}) is functorial in $M$.
It is also functorial in $\{N_{\alpha}\}_{\alpha \in J}$ in the sense that if for each $\alpha \in J$
there is a map of DG-modules $N_{\alpha} \to K_{\alpha}$,
then the diagram
\begin{equation}\label{eqn:nat-in-alpha}
\xymatrix{
\bigoplus_{\alpha \in J} \opn{Hom}_A(M,N_{\alpha}) \ar[r]\ar[d] & \opn{Hom}_A\left(M,\bigoplus_{\alpha \in J} N_{\alpha}\right)\ar[d]\\
\bigoplus_{\alpha \in J} \opn{Hom}_A(M,K_{\alpha}) \ar[r] & \opn{Hom}_A\left(M,\bigoplus_{\alpha \in J} K_{\alpha}\right)
}
\end{equation}
is commutative. 
Recall that if $A$ is a ring and $M$ is a left $A$-module then $M$ is called compact if
for every collection $\{N_{\alpha}\}_{\alpha \in J}$ of left $A$-modules,
the map (\ref{eqn:compact}) is an isomorphism. 
If $A$ is left noetherian then $M$ is compact if and only if it is finitely generated.

One can derive (\ref{eqn:compact}) as follows:
there is a map
\begin{equation}\label{plus-map}
\bigoplus_{\alpha \in J} \mrm{R}\opn{Hom}_A\left(M,N_{\alpha}\right) \to \mrm{R}\opn{Hom}_A\left(M,\bigoplus_{\alpha \in J} N_{\alpha}\right)
\end{equation}
defined by choosing a K-projective resolution
$P \to M$, and letting (\ref{plus-map}) be the composition
\[
\bigoplus_{\alpha \in J} \mrm{R}\opn{Hom}_A\left(M,N_{\alpha}\right) \cong 
\bigoplus_{\alpha \in J} \opn{Hom}_A\left(P,N_{\alpha}\right) \xrightarrow{\alpha}
\opn{Hom}_A\left(P,\bigoplus_{\alpha \in J} N_{\alpha}\right) \cong
\mrm{R}\opn{Hom}_A\left(M,\bigoplus_{\alpha \in J} N_{\alpha}\right)
\]
where $\alpha$ is the map (\ref{eqn:compact}).
Functoriality of (\ref{eqn:compact}) implies that (\ref{plus-map}) is also natural both in $M$ and in $\{N_{\alpha}\}_{\alpha \in J}$.

\begin{lem}\label{lem:is-fg}
Let $A$ be a left noetherian ring,
and let $M$ be a left $A$-module.
Assume that for each collection $\{I_{\alpha}\}_{\alpha \in J}$ of injective left $A$-modules,
the canonical map
\[
\bigoplus_{\alpha \in J} \opn{Hom}_A(M,I_{\alpha}) \to \opn{Hom}_A(M,\bigoplus_{\alpha \in J} I_{\alpha})
\]
is an isomorphism. Then $M$ is finitely generated.
\end{lem}
\begin{proof}
Since over a left noetherian ring the finitely generated modules are exactly the compact modules,
it is enough to show that $M$ is a compact object of $\opn{Mod}(A)$.
Let $\{N_{\alpha}\}_{\alpha \in J}$ be a collection of left $A$-modules.
For each $\alpha \in J$, choose an exact sequence
\[
0 \to N_{\alpha} \to I_{\alpha} \to J_{\alpha}
\]
such that $I_{\alpha}$ and $J_{\alpha}$ are injective left $A$-modules.
We obtain a commutative diagram with exact rows
\[
\xymatrixcolsep{1pc}
\xymatrix{
0 \ar[r] & \bigoplus_{\alpha \in J} \opn{Hom}_A(M,N_{\alpha}) \ar[r]\ar[d] & \bigoplus_{\alpha \in J} \opn{Hom}_A(M,I_{\alpha}) \ar[r] \ar[d] & \bigoplus_{\alpha \in J} \opn{Hom}_A(M,J_{\alpha}) \ar[d]\\
0 \ar[r] & \opn{Hom}_A\left(M,\bigoplus_{\alpha \in J} N_{\alpha}\right) \ar[r]       & \opn{Hom}_A\left(M,\bigoplus_{\alpha \in J} I_{\alpha}\right) \ar[r]        & \opn{Hom}_A\left(M,\bigoplus_{\alpha \in J} J_{\alpha}\right)
}
\]
Since by assumption the two rightmost vertical maps are isomorphisms, we deduce that the natural map
\[
\bigoplus_{\alpha \in J} \opn{Hom}_A(M,N_{\alpha}) \to \opn{Hom}_A(M,\bigoplus_{\alpha \in J} N_{\alpha})
\]
is also an isomorphism. Hence, $M$ is a compact object of $\opn{Mod}(A)$, so it is finitely generated.
\end{proof}

\begin{lem}\label{lem:sum-of-cohoz}
Let $A,B$ be non-positive DG-rings,
and let $\{F_{\alpha}\}_{\alpha \in J}$ be a collection of contravariant triangulated functors of cohomological dimension $0$ from $\cat{D}(A)$ to $\cat{D}(B)$.
Then the contravariant triangulated functor 
\[
\bigoplus_{\alpha \in J} F_{\alpha}: \cat{D}(A) \to \cat{D}(B)
\]
is also of cohomological dimension $0$.
\end{lem}
\begin{proof}
This is clear because cohomology commutes with arbitrary direct sums.
\end{proof}

\begin{thm}\label{thm:bass-papp}
Let $A$ be a non-positive DG-ring.
Then the category $\INJ(A)$ is closed under arbitrary direct sums if and only if
the ring $\mrm{H}^0(A)$ is a left noetherian ring and for each $i < 0$, 
the left $\mrm{H}^0(A)$-module $\mrm{H}^i(A)$ is finitely generated.
\end{thm}
\begin{proof}
Assume first that $\INJ(A)$ is closed under arbitrary direct sums.
Given an arbitrary collection $\{\bar{I}_{\alpha}\}_{\alpha \in J}$ of injective left $\mrm{H}^0(A)$-modules,
for each $\alpha \in J$ there exists by Theorem \ref{thm:eqv} an element $I_{\alpha} \in \INJ(A)$ such that 
\[
\mrm{H}^0(I_{\alpha}) = \bar{I}_{\alpha}.
\]
By assumption, $\bigoplus_{\alpha \in J} I_{\alpha} \in \INJ(A)$.
Hence, by Proposition \ref{prop:HzOfINJ}
\[
\mrm{H}^0\left(\bigoplus_{\alpha \in J} I_{\alpha}\right) = \bigoplus_{\alpha \in J} \mrm{H}^0(I_{\alpha}) = \bigoplus_{\alpha \in J} \bar{I}_{\alpha}
\]
is an injective left $\mrm{H}^0(A)$-module.
The Bass-Papp Theorem now implies that the ring $\mrm{H}^0(A)$ is left noetherian.
Let $n > 0$. 
Since $\bigoplus_{\alpha \in J} I_{\alpha} \in \INJ(A)$,
by Proposition \ref{prop:INJisZeroFunc} the functor $\mrm{R}\opn{Hom}_A(-,\bigoplus_{\alpha \in J} I_{\alpha})$ is of cohomological dimension $0$,
while by Lemma \ref{lem:sum-of-cohoz}, the functor $\bigoplus_{\alpha \in J} \mrm{R}\opn{Hom}_A(-,I_{\alpha})$ is also of cohomological dimension $0$.
Hence, by Theorem \ref{thm:mainZero}, the natural morphism 
\[
\bigoplus_{\alpha \in J} \mrm{R}\opn{Hom}_A(-,I_{\alpha}) \to \mrm{R}\opn{Hom}_A(-,\bigoplus_{\alpha \in J} I_{\alpha})
\]
from (\ref{plus-map}) induces a commutative diagram 
\[
\xymatrix{
\mrm{H}^0\left(\bigoplus_{\alpha \in J} \mrm{R}\opn{Hom}_A(M,I_{\alpha})\right) \ar[r] \ar[d] & \mrm{H}^0\left(\bigoplus_{\alpha \in J} \mrm{R}\opn{Hom}_A(\mrm{H}^0(M),I_{\alpha})\right)\ar[d]\\
\mrm{H}^0\left(\mrm{R}\opn{Hom}_A(M,\bigoplus_{\alpha \in J} I_{\alpha})\right) \ar[r]        & \mrm{H}^0\left(\mrm{R}\opn{Hom}_A(\mrm{H}^0(M),\bigoplus_{\alpha \in J} I_{\alpha})\right)
}
\]
for every $M \in \cat{D}(A)$, such that the horizontal maps are isomorphisms.
Using (\ref{plus-map}) twice and adjunction, there is a commutative diagram
\[
\xymatrixcolsep{1pc}
\xymatrix{
\bigoplus_{\alpha \in J} \mrm{R}\opn{Hom}_A(\mrm{H}^0(M),I_{\alpha})\ar[dd]\ar[r] & \bigoplus_{\alpha \in J} \mrm{R}\opn{Hom}_{\mrm{H}^0(A)}(\mrm{H}^0(M),\mrm{R}\opn{Hom}_A(\mrm{H}^0(A),I_{\alpha})) \ar[d] \\
& \mrm{R}\opn{Hom}_{\mrm{H}^0(A)}(\mrm{H}^0(M),\bigoplus_{\alpha \in J} \mrm{R}\opn{Hom}_A(\mrm{H}^0(A),I_{\alpha})) \ar[d] \\
\mrm{R}\opn{Hom}_A(\mrm{H}^0(M),\bigoplus_{\alpha \in J} I_{\alpha}) \ar[r] & \mrm{R}\opn{Hom}_{\mrm{H}^0(A)}(\mrm{H}^0(M), \mrm{R}\opn{Hom}_A(\mrm{H}^0(A),\bigoplus_{\alpha \in J} I_{\alpha}))
}
\]
in which the horizontal maps are isomorphisms.
By Proposition \ref{prop:RHomisHZ} and (\ref{eqn:nat-in-alpha}) there is a commutative diagram
\[
\xymatrixcolsep{1pc}
\xymatrix{
\bigoplus_{\alpha \in J} \mrm{R}\opn{Hom}_{\mrm{H}^0(A)}(\mrm{H}^0(M),\mrm{R}\opn{Hom}_A(\mrm{H}^0(A),I_{\alpha})) \ar[d] \ar[r] & \bigoplus_{\alpha \in J} \mrm{R}\opn{Hom}_{\mrm{H}^0(A)}(\mrm{H}^0(M),\mrm{H}^0(I_{\alpha}))\ar[d]\\
\mrm{R}\opn{Hom}_{\mrm{H}^0(A)}(\mrm{H}^0(M),\bigoplus_{\alpha \in J} \mrm{R}\opn{Hom}_A(\mrm{H}^0(A),I_{\alpha})) \ar[d] \ar[r] & \mrm{R}\opn{Hom}_{\mrm{H}^0(A)}(\mrm{H}^0(M),\bigoplus_{\alpha \in J} \mrm{H}^0(I_{\alpha}))\ar[d]\\
\mrm{R}\opn{Hom}_{\mrm{H}^0(A)}(\mrm{H}^0(M), \mrm{R}\opn{Hom}_A(\mrm{H}^0(A),\bigoplus_{\alpha \in J} I_{\alpha})) \ar[r] & \mrm{R}\opn{Hom}_{\mrm{H}^0(A)}(\mrm{H}^0(M), \mrm{H}^0(\bigoplus_{\alpha \in J} I_{\alpha}))
}
\]
in which the horizontal maps are isomorphisms.
Combining the last three commutative diagrams, 
we see that for any $M \in \cat{D}(A)$ there is a commutative diagram in which the horizontal maps are isomorphisms:
\[
\xymatrix{
\mrm{H}^0\left(\bigoplus_{\alpha \in J} \mrm{R}\opn{Hom}_A(M,I_{\alpha})\right) \ar[r] \ar[d] & \bigoplus_{\alpha \in J} \mrm{R}\opn{Hom}_{\mrm{H}^0(A)}(\mrm{H}^0(M),\mrm{H}^0(I_{\alpha}))\ar[d]\\
\mrm{H}^0\left(\mrm{R}\opn{Hom}_A(M,\bigoplus_{\alpha \in J} I_{\alpha})\right) \ar[r]        & \mrm{R}\opn{Hom}_{\mrm{H}^0(A)}(\mrm{H}^0(M), \mrm{H}^0(\bigoplus_{\alpha \in J} I_{\alpha}))
}
\]
Taking $M = A[-n]$, we see that the leftmost vertical map is an isomorphism.
Hence, the natural map
\[
\bigoplus_{\alpha \in J} \opn{Hom}_{\mrm{H}^0(A)}(\mrm{H}^{-n}(A),\bar{I}_{\alpha}) \to \opn{Hom}_{\mrm{H}^0(A)}(\mrm{H}^{-n}(A),\bigoplus_{\alpha \in J} \bar{I}_{\alpha})
\]
is an isomorphism. It follows from Lemma \ref{lem:is-fg} the left $\mrm{H}^0(A)$-module $\mrm{H}^{-n}(A)$ is finitely generated.

Conversely, assume that $\mrm{H}^0(A)$ is a left noetherian ring and for each $i < 0$, 
the left $\mrm{H}^0(A)$-module $\mrm{H}^i(A)$ is finitely generated.
Let $\{I_{\alpha}\}_{\alpha \in J}$ be a collection of elements of $\INJ(A)$.
Let $h_{\alpha}: I_{\alpha} \inj \bigoplus_{\alpha \in J} I_{\alpha}$ be the canonical inclusion.
By Proposition \ref{prop:HzOfINJ}, each $\mrm{H}^0(I_{\alpha})$ is an injective $\mrm{H}^0(A)$-module.
By the Bass-Papp Theorem 
\[
\bigoplus_{\alpha \in J} \mrm{H}^0(I_{\alpha})
\]
is also an injective $\mrm{H}^0(A)$-module.
Hence, by Theorem \ref{thm:eqv} there is a (unique up to isomorphism) DG-module $K \in \INJ(A)$ such that
\[
\mrm{H}^0(K) = \bigoplus_{\alpha \in J} \mrm{H}^0(I_{\alpha}).
\]
Since for each $\alpha \in J$ we have an inclusion map 
\[
\bar{f}_{\alpha}:\mrm{H}^0(I_{\alpha}) \to \mrm{H}^0(K),
\]
by Theorem \ref{thm:eqv} it lifts to a map $f_{\alpha}:I_{\alpha} \to K$ such that $\mrm{H}^0(f_{\alpha}) = \bar{f}_{\alpha}$.
It follows by the universal property of direct sums that there is a unique map 
\[
f:\bigoplus_{\alpha \in J} I_{\alpha} \to K
\]
such that for any $\alpha \in J$ the diagram
\begin{equation}\label{eqn:hfKdiag}
\xymatrixcolsep{4pc}
\xymatrix{
\bigoplus_{\alpha \in J} I_{\alpha} \ar[r]^f & K\\
I_{\alpha} \ar[u]^{h_{\alpha}} \ar[ur]_{f_{\alpha}}
}
\end{equation}
is commutative. 
Since $K \in \INJ(A)$ and $\INJ(A)$ is closed under isomorphisms,
it is thus enough to show that $f$ is an isomorphism,
and to show this it is enough to show that for every $n \in \mathbb{Z}$,
the map 
\[
\mrm{H}^n(f):\mrm{H}^n\left(\bigoplus_{\alpha \in J} I_{\alpha}\right) \to \mrm{H}^n(K)
\]
is an isomorphism.
For $n < 0$ both sides are $0$ so this map is an isomorphism,
while for $n = 0$ this map is an isomorphism by the definitions of $K$ and $f$.
Let $n > 0$. For every $\alpha \in J$,
by Corollary \ref{cor:coho-of-INJ} there is a commutative diagram
\[
\xymatrix{
\mrm{H}^n(I_{\alpha}) \ar[r]^{\mrm{H}^n(f_{\alpha})}\ar[d]^{\rotatebox{90}{$\cong$}} & \mrm{H}^n(K)\ar[d]^{\rotatebox{90}{$\cong$}}\\
\opn{Hom}_{\mrm{H}^0(A)}\left(\mrm{H}^{-n}(A),\mrm{H}^0(I_{\alpha})\right) \ar[r] & \opn{Hom}_{\mrm{H}^0(A)}\left(\mrm{H}^{-n}(A),\mrm{H}^0(K)\right)
}
\]
such that the vertical maps are isomorphisms. Combining all these maps we obtain a commutative diagram
\begin{equation}\label{eqn:combined-diag}
\xymatrix{
\bigoplus_{\alpha \in J} \mrm{H}^n(I_{\alpha}) \ar[r]^{\psi}\ar[d]^{\rotatebox{90}{$\cong$}} & \mrm{H}^n(K)\ar[d]^{\rotatebox{90}{$\cong$}}\\
\bigoplus_{\alpha \in J} \opn{Hom}_{\mrm{H}^0(A)}\left(\mrm{H}^{-n}(A),\mrm{H}^0(I_{\alpha})\right) \ar[r] & \opn{Hom}_{\mrm{H}^0(A)}\left(\mrm{H}^{-n}(A),\mrm{H}^0(K)\right)
}
\end{equation}
Since by definition the map $\mrm{H}^0(I_{\alpha}) \to \mrm{H}^0(K) = \bigoplus_{\alpha \in J} \mrm{H}^0(I_{\alpha})$
is the inclusion map $\bar{f}_{\alpha}$,
we see that the bottom horizontal map of (\ref{eqn:combined-diag}) is of the form of (\ref{eqn:compact}).
Since $\mrm{H}^0(A)$ is left noetherian and $\mrm{H}^{-n}(A)$ is finitely generated, 
we deduce that it is an isomorphism.
Hence, the top horizontal map of (\ref{eqn:combined-diag}) (which we denoted by $\psi$) is also an isomorphism.
The map $\psi$ fits into the following diagram
\begin{equation}\label{diag-with-eps}
\xymatrixcolsep{3pc}
\xymatrixrowsep{3pc}
\xymatrix{
\bigoplus_{\alpha \in J} \mrm{H}^n(I_{\alpha}) \ar[r]^{\psi}\ar[d]_{\epsilon} & \mrm{H}^n(K)\\
\mrm{H}^n(\bigoplus_{\alpha \in J} I_{\alpha}) \ar[ru]_{\mrm{H}^n(f)}
}
\end{equation}
Here, the map denoted by $\epsilon$ is the canonical isomorphism induced from the maps 
\[
\mrm{H}^n(h_{\alpha}):\mrm{H}^n(I_{\alpha}) \to \mrm{H}^n\left(\bigoplus_{\alpha \in J} I_{\alpha}\right).
\]
We claim that (\ref{diag-with-eps}) commutes. To see this, for each $\alpha \in J$,
let 
\[
g_{\alpha}: \mrm{H}^n(I_{\alpha}) \inj \bigoplus_{\alpha \in J} \mrm{H}^n(I_{\alpha})
\]
be the canonical inclusion.
Then by the universal property of direct sums, it is enough to show that 
\begin{equation}\label{eqn:last-eqn}
\mrm{H}^n(f) \circ \epsilon \circ g_{\alpha} = \psi \circ g_{\alpha}.
\end{equation}
On the one hand, by the definition of $\psi$ we have $\psi \circ g_{\alpha} = \mrm{H}^n(f_{\alpha})$.
On the other hand, by the definition of $\epsilon$ we have $\epsilon \circ g_{\alpha} = \mrm{H}^n(h_{\alpha})$.
By commutativity of (\ref{eqn:hfKdiag}) we see that 
\[
\mrm{H}^n(f) \circ \mrm{H}^n(h_{\alpha}) = \mrm{H}^n(f \circ h_{\alpha}) = \mrm{H}^n(f_{\alpha}),
\]
so that (\ref{eqn:last-eqn}) holds.
We have thus seen that (\ref{diag-with-eps}) commutes.
Since $\epsilon$ and $\psi$ are isomorphisms,
we deduce that $\mrm{H}^n(f)$ is an isomorphism, 
which implies that $f$ is an isomorphism, so that 
\[
\bigoplus_{\alpha \in J} I_{\alpha} \in \INJ(A),
\]
as claimed.
\end{proof}

In view of this result, it makes sense to define:

\begin{dfn}\label{def:noeth}
Let $A$ be a non-positive DG-ring.
We say that $A$ is left noetherian if
the ring $\mrm{H}^0(A)$ is a left noetherian ring and for each $i < 0$, 
the left $\mrm{H}^0(A)$-module $\mrm{H}^i(A)$ is finitely generated.
\end{dfn}

\begin{cor}\label{cor:sum-of-ind}
Let $A$ be a left noetherian non-positive DG-ring.
Then for any DG-module $I \in \INJ(A)$ there is an isomorphism
\[
I \cong \bigoplus_{\alpha \in J} I_{\alpha}
\]
such that for each $\alpha \in J$, we have that $I_{\alpha} \in \INJ(A)$ and that $I_{\alpha}$ is an indecomposable element of $\cat{D}(A)$.
\end{cor}
\begin{proof}
By \cite[Theorem 2.5]{Ma}, we may write
\[
\mrm{H}^0(I) = \bigoplus_{\alpha \in J} \bar{I}_{\alpha}
\]
such that each $\bar{I}_{\alpha}$ is an indecomposable injective $\mrm{H}^0(A)$-module.
By Theorem \ref{thm:eqv}, for each $\alpha \in J$ there exists $I_{\alpha} \in \INJ(A)$ such that 
$\mrm{H}^0(I_{\alpha}) = \bar{I}_{\alpha}$, and moreover $I_{\alpha}$ is an indecomposable element of $\cat{D}(A)$.
Let 
\[
K := \bigoplus_{\alpha \in J} I_{\alpha} \in \cat{D}(A).
\]
By Theorem \ref{thm:bass-papp}, $K \in \INJ(A)$.
Since 
\[
\mrm{H}^0(K) = \mrm{H}^0\left(\bigoplus_{\alpha \in J} I_{\alpha}\right) = \bigoplus_{\alpha \in J} \left(\mrm{H}^0(I_{\alpha})\right) \cong \mrm{H}^0(I)
\]
we deduce from Theorem \ref{thm:eqv} that there is an isomorphism $I \cong K$, proving the claim.
\end{proof}

If $A$ is a left noetherian ring, 
then by \cite[Theorem 19.10]{La},
the left injective module 
\[
E := \bigoplus_{V_i \text{ is a simple left } A \text{-module}} E(V_i)
\]
is a cogenerator of $A$. Here $E(V_i)$ is the injective hull of $V_i$.
This means that for any non-zero left $A$-module $M$, 
one has $\opn{Hom}_A(M,E) \ne 0$.
The $A$-module $E$ is called the minimal injective cogenerator of $A$.

Recall that an element $X$ of a triangulated category $\mathcal{T}$ is called a cogenerator if for any $0 \ne M \in \mathcal{T}$ there is some $n \in \mathbb{Z}$,
such that 
\[
\opn{Hom}_{\mathcal{T}}(M,X[n]) \ne 0.
\]

\begin{cor}\label{cor:cogen}
Let $A$ be a left noetherian non-positive DG-ring.
Let $E \in \INJ(A)$ be the DG-module such that $\mrm{H}^0(E)$ is the minimal injective cogenerator of $\mrm{H}^0(A)$.
Then $E$ is a cogenerator of $\cat{D}(A)$.
\end{cor}
\begin{proof}
Let $0 \ne M \in \cat{D}(A)$.
Then there is some $n \in \mathbb{Z}$ such that $\mrm{H}^n(M) \ne 0$.
By Theorem \ref{thm:coho-of-rhom} we have that
\[
\mrm{H}^{-n} \left(\mrm{R}\opn{Hom}_A(M,E)\right) \cong \opn{Hom}_{\mrm{H}^0(A)}(\mrm{H}^n(M),\mrm{H}^0(E)),
\]
and the latter is non-zero because $\mrm{H}^0(E)$ is a cogenerator.
\end{proof}

\section{Injectives over commutative noetherian DG-rings}\label{sec:comm}

\subsection{Some preliminaries about commutative noetherian DG-rings}\label{sec:prelcomm}

A DG-ring $A$ is called commutative if for all $a \in A^i$ and $b \in A^j$,
one has $b\cdot a = (-1)^{i\cdot j} \cdot a \cdot b$, and $a^2 = 0$ if $i$ is odd.
If $A$ is a commutative non-positive DG-ring which is left noetherian in the sense of Definition \ref{def:noeth} we will say that $A$ is a \textit{commutative noetherian DG-ring}.
Thus, by definition, a commutative noetherian DG-ring will be assumed non-positive.
In that case, $\mrm{H}^0(A)$ is a commutative noetherian ring.
Under this assumption, we denote by $\cat{D}_{\mrm{f}}(A)$ the full triangulated subcategory of $\cat{D}(A)$
consisting of DG-modules $M$ such that $\mrm{H}^n(M)$ is a finitely generated over $\mrm{H}^0(A)$ for all $n \in \mathbb{Z}$.
Similarly, we set $\cat{D}^{-}_{\mrm{f}}(A) = \cat{D}^{-}(A) \cap \cat{D}_{\mrm{f}}(A)$.
Note that the noetherian condition implies that $A \in \cat{D}^{-}_{\mrm{f}}(A)$.

\subsubsection{Localization of commutative DG-rings}

Let $A$ be a commutative non-positive DG-ring.
Given a multiplicatively closed set $\bar{S} \subseteq \mrm{H}^0(A)$,
recall that the localization $\bar{S}^{-1} A$ is defined as follows:
let $S \subseteq A^0$ be the inverse image of $\bar{S}$ along the surjection $A^0 \to \mrm{H}^0(A)$,
and define 
\[
\bar{S}^{-1} A := A\otimes_{A^0} S^{-1}A^0.
\]
The localization map $A^0 \to S^{-1}A^0$ induces a localization map $A \to \bar{S}^{-1}A$.
For $\p \in \opn{Spec}(\mrm{H}^0(A))$, we will write as usual $A_{\p}$ instead of $(\mrm{H}^0(A)-\p)^{-1}A$.
If $M$ is a DG-module over $A$, then we define
\[
\bar{S}^{-1}(M) := M\otimes_{A^0} S^{-1}A^0 \in \cat{D}(\bar{S}^{-1}A).
\]
Again, if $\p \in \opn{Spec}(\mrm{H}^0(A))$, we will write $M_{\p}$ instead of $(\mrm{H}^0(A)-\p)^{-1}M$.

A commutative noetherian DG-ring $A$ is called a local noetherian DG-ring if the commutative noetherian ring $\mrm{H}^0(A)$ is a local ring.
If $\m \subseteq \mrm{H}^0(A)$ is its maximal ideal, one says that $(A,\m)$ is a commutative noetherian local DG-ring.
Letting $\k = (\mrm{H}^0(A)/\m)$ be the residue field, we will sometimes say that $(A,\m,\k)$ is a commutative noetherian local DG-ring.

\subsubsection{Local cohomology and local homology over commutative DG-rings}\label{sec:localcoh}

We now review the notions of local cohomology and local homology over commutative DG-rings.
A reference for this is material is \cite[Section 2]{Sh2}.
See also \cite{BIK}.

Let $A$ be a commutative non-positive DG-ring,
and let $\a \subseteq \mrm{H}^0(A)$ be a finitely generated ideal.
The category of derived $\a$-torsion DG-modules over $A$,
denoted by $\cat{D}_{\a-\opn{tor}}(A)$,
is the full triangulated subcategory of $D(A)$ consisting of DG-modules $M$
such that the $\mrm{H}^0(A)$-module $\mrm{H}^n(M)$ is $\a$-torsion for all $n\in \mathbb{Z}$.
The inclusion functor
\[
F:\cat{D}_{\a-\opn{tor}}(A) \inj \cat{D}(A)
\]
has a right adjoint
\[
G: \cat{D}(A) \to \cat{D}_{\a-\opn{tor}}(A).
\]
One defined the local cohomology functor of $A$ with respect to $\a$ to be the composition 
\[
\mrm{R}\Gamma_{\a}(-) := F \circ G(-): \cat{D}(A) \to \cat{D}(A).
\]
This coincides with the usual local cohomology functor if $A$ is a commutative noetherian ring.

For any $M \in \cat{D}(A)$,
there is a natural map 
\begin{equation}
\sigma_M:\mrm{R}\Gamma_{\a}(M) \to M,
\end{equation}
and it holds that $M \in \cat{D}_{\a-\opn{tor}}(A)$ if and only if $\sigma_M$ is an isomorphism.

The functor $\mrm{R}\Gamma_{\a}$ has a right adjoint which we denote by 
\[
\mrm{L}\Lambda_{\a}(-) : \cat{D}(A) \to \cat{D}(A).
\]
This functor is called the local homology or derived completion functor with respect to $\a$.
If $A$ is a commutative noetherian ring it coincides with the left derived functor of the $\a$-adic completion functor.

The Greenlees-May duality, 
which originated in \cite{GM},
states that for any commutative DG-ring $A$,
and any finitely generated ideal $\a\subseteq \mrm{H}^0(A)$, 
there are bifunctorial isomorphisms
\begin{eqnarray}
\mrm{R}\opn{Hom}_A(\mrm{R}\Gamma_{\a}(M),N) \cong
\mrm{R}\opn{Hom}_A(\mrm{R}\Gamma_{\a}(M),\mrm{R}\Gamma_{\a}(N)) \cong\nonumber\\
\mrm{R}\opn{Hom}_A(\mrm{L}\Lambda_{\a}(M),\mrm{L}\Lambda_{\a}(N)) \cong
\mrm{R}\opn{Hom}_A(M,\mrm{L}\Lambda_{\a}(N))\nonumber
\end{eqnarray}
for any $M,N \in \cat{D}(A)$.
The Matlis-Greenlees-May (MGM) equivalence states that there are isomorphisms
\[
\mrm{R}\Gamma_{\a} \circ \mrm{L}\Lambda_{\a}(-) \cong \mrm{R}\Gamma_{\a}(-)
\]
and
\[
\mrm{L}\Lambda_{\a} \circ \mrm{R}\Gamma_{\a} (-) \cong \mrm{L}\Lambda_{\a}(-)
\]
of functors $\cat{D}(A) \to \cat{D}(A)$.

\begin{rem}
If $A$ is a commutative non-positive DG-ring,
and $\a\subseteq \mrm{H}^0(A)$ is a finitely generated ideal,
the one can consider the local cohomology functor of $\mrm{H}^0(A)$ with respect to $\a$,
and the local cohomology functor of $A$ with respect to $\a$.
The former is a functor $\cat{D}(\mrm{H}^0(A)) \to \cat{D}(\mrm{H}^0(A))$,
and the latter is a functor $\cat{D}(A) \to \cat{D}(A)$.
According to the notation introduced above,
both of these functors are denoted by $\mrm{R}\Gamma_{\a}$.
In some cases we will need to use both functors.
To resolve the ambiguity in notation in these cases,
we will sometimes denote the former by
\[
\mrm{R}\Gamma_{\a}^{\mrm{H}^0(A)}:\cat{D}(\mrm{H}^0(A)) \to \cat{D}(\mrm{H}^0(A)),
\]
and the latter by
\[
\mrm{R}\Gamma_{\a}^{A}:\cat{D}(A) \to \cat{D}(A).
\]
A similar remark applies to the functor $\mrm{L}\Lambda_{\a}$.
\end{rem}

\subsubsection{The telescope complex and the telescope DG-module}

A reference for the facts of this subsection is \cite[Section 5]{PSY1}.
Given a commutative ring $A$ and an element $a \in A$,
we define the telescope complex $\opn{Tel}(A;a) \in \cat{D}(A)$ to be the complex
\[
0 \to \bigoplus_{i=0}^\infty A \xrightarrow{d} \bigoplus_{i=0}^\infty A \to 0
\]
with non-zero components in degrees $0,1$. Its differential $d$ is defined as follows:
Let $\{\delta_i\mid i\ge 0\}$ be the basis of the countably generated free A-module 
$\bigoplus_{i=0}^\infty A$. Then
\[
d(\delta_i) = \begin{cases}
               \delta_0 & \text{if $i=0$}\\
               \delta_{i-1}-a\cdot \delta_i & \text{if $i\ge 1$}
              \end{cases}
\]

For a finite sequence $\mathbf{a} = (a_1,\dots,a_n)$ of elements of $A$,
one defines:
\[
\opn{Tel}(A;\mathbf{a}) := \opn{Tel}(A;a_1) \otimes_A \opn{Tel}(A;a_2) \otimes_A \opn{Tel}(A;a_3) \otimes_A \dots \otimes_A \opn{Tel}(A;a_n).
\]
This is a bounded complex of countably generated free $A$-modules.
In particular, it is K-projective.

Let $A$ be a commutative ring, let $\a\subseteq A$ be a finitely generated ideal,
and let $\mathbf{a}$ be a finite sequence of elements of $A$ that generates $\a$.
According to \cite[Section 5]{PSY1}, there is a map 
\begin{equation}\label{eqn:toTele}
u_{\mathbf{a}}:\opn{Tel}(A;\mathbf{a}) \to A
\end{equation}
and a morphism of functors
\[
\mrm{R}\Gamma_{\a}(M) \to \opn{Tel}(A;\mathbf{a}) \otimes_A M
\]
If $A$ is noetherian, then there is a commutative diagram
\begin{equation}\label{eqn:diag-of-rgamma}
\xymatrix{
\mrm{R}\Gamma_{\a}(M) \ar[r]^{\sigma_M}\ar[d] & M\\
\opn{Tel}(A;\mathbf{a})\otimes_A M \ar[ru]_{u_{\mathbf{a}}\otimes_A 1_M} &
}
\end{equation}
in $\cat{D}(A)$, such that the vertical map is an isomorphism.

Let $A$ be a commutative ring, 
let $\a\subseteq A$ be a finitely generated ideal,
and let $\mathbf{a}$ be a finite sequence of elements of $A$ that generates $\a$.
Let $B$ be another commutative ring,
let $f:A \to B$ be a ring homomorphism,
and let $\mathbf{b} = f(\mathbf{a})$.
Then the telescope complex satisfies the base change property:
there is an isomorphism of complexes of $B$-modules:
\[
\opn{Tel}(A;\mathbf{a})\otimes_A B \cong \opn{Tel}(B;\mathbf{b}).
\]

In view of this fact, given a commutative non-positive DG-ring $A$,
and a finite sequence of elements $\mathbf{a}$ of $A^0$,
it makes sense to set
\[
\opn{Tel}(A;\mathbf{a}) := \opn{Tel}(A^0;\mathbf{a})\otimes_{A^0} A.
\]
Tensoring (\ref{eqn:toTele}) with $A$ over $A^0$, we obtain a morphism
\begin{equation}\label{eqn:toTeleA}
\opn{Tel}(A;\mathbf{a}) \to A
\end{equation}

Given a commutative non-positive DG-ring $A$, a finitely generated ideal $\a\subseteq \mrm{H}^0(A)$,
and a finite sequence $\mathbf{a}$ of elements of $A^0$ whose image in $\mrm{H}^0(A)$ generate $\a$,
it is shown in \cite[Section 2]{Sh2} that there is an isomorphism
\begin{equation}\label{eqn:RGammaIso}
\mrm{R}\Gamma_{\a}(-) \cong \opn{Tel}(A;\mathbf{a})\otimes_A -
\end{equation}
of functors $\cat{D}(A)\to \cat{D}(A)$.

\subsubsection{Derived completion of commutative DG-rings}

Given a commutative non-positive DG-ring $A$
and a finitely generated ideal $\a\subseteq \mrm{H}^0(A)$,
according to \cite[Section 3]{Sh2}, 
there is a commutative non-positive DG-ring denoted by $\mrm{L}\Lambda(A,\a)$ which is called the derived $\a$-adic completion of $A$.
The construction is functorial in a suitable homotopy category.
There is a natural map $A \to \mrm{L}\Lambda(A,\a)$, 
but it is only defined in that homotopy category.
It induces a forgetful functor
\[
Q:\cat{D}(\mrm{L}\Lambda(A,\a)) \to \cat{D}(A).
\]
It can be described more explicitly as follows:
there is a commutative non-positive DG-ring $P$,
and a diagram of DG-ring homomorphisms
\begin{equation}\label{eqn:completion-diag}
A \xleftarrow{\phi} P \xrightarrow{\psi} \mrm{L}\Lambda(A,\a)
\end{equation}
such that the map $\phi$ is a quasi-isomorphism.
Hence, there is an equivalence of triangulated categories $\cat{D}(A) \cong \cat{D}(P)$,
and $Q$ is the composition
\[
\cat{D}(\mrm{L}\Lambda(A,\a)) \to \cat{D}(P) \cong \cat{D}(A)
\]
where the first functor is the ordinary forgetful functor along the DG-ring homomorphism $\psi$.
By \cite[Proposition 3.58]{Sh2}, one has $Q(\widehat{A}) = \mrm{L}\Lambda_{\a}(A)$.
If $\psi$ is also quasi-isomorphisms, 
one says that $A$ is cohomologically $\a$-adically complete.
If $A$ is noetherian then $\mrm{L}\Lambda(A,\a)$ is also noetherian.
In that case 
\[
\mrm{H}^0(\mrm{L}\Lambda(A,\a)) = \Lambda_{\a}(\mrm{H}^0(A)),
\]
where we have denoted by $\Lambda_{\a}(-)$ the classical $\a$-adic completion functor.
In particular, if $(A,\m)$ is a commutative noetherian local DG-ring 
then $\mrm{L}\Lambda(A,\m)$ is also a commutative noetherian local DG-ring.

According to \cite[Section 5]{Sh2}, the functors $\mrm{R}\Gamma_{\a}, \mrm{L}\Lambda_{\a}$ lift to $\cat{D}(\mrm{L}\Lambda(A,\a))$.
Precisely, there are triangulated functors
\[
\mrm{R}\widehat{\Gamma}_{\a}(-), \mrm{L}\widehat{\Lambda}_{\a}(-):\cat{D}(A) \to \cat{D}(\mrm{L}\Lambda(A,\a))
\]
such that there are isomorphisms
\begin{equation}
Q \circ \mrm{R}\widehat{\Gamma}_{\a}(-) \cong \mrm{R}\Gamma_{\a}(-)
\end{equation}
and
\begin{equation}\label{eqn:whlambda}
Q \circ \mrm{L}\widehat{\Lambda}_{\a}(-) \cong \mrm{L}\Lambda_{\a}(-)
\end{equation}
of functors $\cat{D}(A) \to \cat{D}(A)$.
See \cite{SW} and \cite[Section 3]{ShHC} for a study of these functors over commutative rings.

\subsubsection{The tensor-evaluation isomorphism}

We now wish to discuss a DG version of the tensor evaluation morphism. 
Before that, we must discuss the notion of flat dimension of a DG-module.
Similarly to the discussion in Section \ref{sec:injdims}, given a commutative non-positive DG-ring $A$ and $M \in \cat{D}(A)$,
let us define the bounded flat dimension of $M$ over $A$ to be the number
\[
\inf \{ n \in \mathbb{Z} \mid \Tor^i_A(N,M) = 0 \text{ for any } N\in \cat{D}^{\mrm{b}}(A) \text{ and any } i>n-\inf N\}
\] 
where $\Tor^i_A(N,M) := \mrm{H}^{-i}(N\otimes^{\mrm{L}}_A M)$.
Let us denote it by $\bflatdim_A(M)$.
Note that in \cite{Sh3}, we simply called it the flat dimension of $M$.
The unbounded flat dimension of $M$ over $A$ is the number
\[
\inf \{ n \in \mathbb{Z} \mid \Tor^i_A(N,M) = 0 \text{ for any } N\in \cat{D}^{+}(A) \text{ and any } i>n-\inf N\}
\] 
which we denote by $\uflatdim_A(M)$. 
Similarly to Theorem \ref{thm:buINJ}, we expect that the equality $\bflatdim_A(M) = \uflatdim_A(M)$ will always hold.
Since we will not need this equality in this paper, we will not study this question here.

Here is a version of the tensor-evaluation morphism:
\begin{prop}\label{prop:eval}
Let $A$ be a commutative noetherian DG-ring.
Then for any $M \in \cat{D}^{-}_{\mrm{f}}(A)$, $N \in \cat{D}^{+}(A)$ and $K \in \cat{D}(A)$ such that $\uflatdim_A(K) < \infty$,
there is an isomorphism
\begin{equation}\label{eqn:tev}
\mrm{R}\opn{Hom}_A(M,N)\otimes^{\mrm{L}}_A K \to \mrm{R}\opn{Hom}_A(M,N\otimes^{\mrm{L}}_A K) 
\end{equation}
in $\cat{D}(A)$ which is functorial in $M,N,K$.
\end{prop}
\begin{proof}
This was shown in \cite[Proposition 2.2(1)]{Sh3}
(see also \cite[Proposition 6.7]{Sh1}).
There, we assumed the stronger condition $N \in \cat{D}^{\mrm{b}}(A)$ and the (possible) weaker condition that $\bflatdim_A(K) < \infty$.
Under the stronger assumption that we make here that $\uflatdim_A(K) < \infty$, the same proof works for $N \in \cat{D}^{+}(A)$.
\end{proof}

\subsection{Injectives over commutative noetherian rings}

Let $A$ be a commutative noetherian DG-ring.
In this case, $\mrm{H}^0(A)$ is a commutative noetherian ring.
Hence, by \cite[Proposition 3.1]{Ma}, 
there is a bijection between indecomposable injective $\mrm{H}^0(A)$-modules and elements of $\opn{Spec}(\mrm{H}^0(A))$.
It follows by Theorem \ref{thm:eqv} that there is a similar bijection between indecomposable elements of $\INJ(A)$ and elements of $\opn{Spec}(\mrm{H}^0(A))$.
Given $\p \in \opn{Spec}(\mrm{H}^0(A))$, we will denote by $E(A,\p)$ the element of $\INJ(A)$ corresponding to $\p$.
Thus, $E(A,\p)$ is defined to be the unique (up to isomorphism) element of $\INJ(A)$ such that 
\[
\mrm{H}^0\left(E(A,\p)\right) = E(\mrm{H}^0(A),\p)
\]
is the injective hull (over $\mrm{H}^0(A)$) of the residue field 
\[
k(\p) = (\mrm{H}^0(A)_{\p})/(\p\cdot (\mrm{H}^0(A)_{\p})).
\]
By Corollary \ref{cor:sum-of-ind}, every element of $\INJ(A)$ is isomorphic to a direct sum of DG-modules of the form $E(A,\p)$ where $\p \in \opn{Spec}(\mrm{H}^0(A))$.
The aim of this section is to make a study of the DG-module $E(A,\p)$.

\begin{prop}\label{prop:localz}
Let $A$ be a commutative noetherian DG-ring,
and let $\p \in \opn{Spec}(\mrm{H}^0(A))$.
Then the localization map
\[
E(A,\p) \to \left(E(A,\p)\right)_{\p}
\]
is an isomorphism.
Hence, denoting by $Q_{\p}:\cat{D}(A_{\p}) \to \cat{D}(A)$ the forgetful functor,
one has $Q_{\p}(E(A_{\p},\p)) \cong E(A,\p)$.
\end{prop}
\begin{proof}
The localization morphism $A \to A_{\p}$ and naturality of (\ref{eqn:tev})
imply that there is a commutative diagram
\[
\xymatrix{
\mrm{R}\opn{Hom}_A\left(\mrm{H}^0(A),E(A,\p)\right) \otimes_A A \ar[r]\ar[d] & \mrm{R}\opn{Hom}_A\left(\mrm{H}^0(A),E(A,\p)\otimes_A A\right) \ar[d]\\
\mrm{R}\opn{Hom}_A\left(\mrm{H}^0(A),E(A,\p)\right) \otimes_A A_{\p}  \ar[r] & \mrm{R}\opn{Hom}_A\left(\mrm{H}^0(A),E(A,\p)\otimes_A A_{\p}\right)
}
\]
The top horizontal map is clearly an isomorphism.
Since $\uflatdim_A(A_{\p}) < \infty$,
by Proposition \ref{prop:eval} the bottom horizontal map is also an isomorphism.
By definition of $E(A,\p)$, we have that 
\[
\mrm{R}\opn{Hom}_A\left(\mrm{H}^0(A),E(A,\p)\right) \cong E\left(\mrm{H}^0(A),\p\right),
\]
and as is known, the latter is an $\mrm{H}^0(A)_{\p}$-module.
Hence, the leftmost vertical map is an isomorphism.
We deduce that the map
\[
\mrm{R}\opn{Hom}_A\left(\mrm{H}^0(A),E(A,\p)\otimes_A A\right) \to \mrm{R}\opn{Hom}_A\left(\mrm{H}^0(A),E(A,\p)\otimes_A A_{\p}\right)
\]
is an isomorphism,
so by Proposition \ref{prop:RHomReflect}, the localization map
\[
E(A,\p) \to \left(E(A,\p)\right)_{\p}
\]
is an isomorphism.
Letting $E = \left(E(A,\p)\right)_{\p} \in \cat{D}^{+}(A_{\p})$, 
it is thus enough to show that $E \in \INJ(A_{\p})$,
and this follows from the adjunction
\[
\mrm{R}\opn{Hom}_{A_{\p}}(\mrm{H}^0(A_{\p}),E) \cong \mrm{R}\opn{Hom}_A(\mrm{H}^0(A),E) \cong \mrm{R}\opn{Hom}_A(\mrm{H}^0(A),E(A,\p))
\]
and Proposition \ref{prop:INJCHAR}.
\end{proof}

\begin{prop}\label{prop:lcofinj}
Let $A$ be a commutative noetherian DG-ring,
let $\p \in \opn{Spec}(\mrm{H}^0(A))$,
and let $\a \subseteq \mrm{H}^0(A)$ be an ideal.
Then one has
\[
\mrm{R}\Gamma_{\a}\left(E(A,\p)\right) = \begin{cases}
                                          E(A,\p) & \text{ if } \a \subseteq \p\\
                                          0       & \text{ otherwise.}
                                         \end{cases}
\]
\end{prop}
\begin{proof}
To shorten notation, let us set $E := E(A,\p)$.
Let $\mathbf{a}$ be a finite sequence of elements of $A^0$ such that their images in $\mrm{H}^0(A)$ generate the ideal $\a$.
Denote this image by $\bar{\mathbf{a}}$.
The map $\opn{Tel}(A;\mathbf{a}) \to A$ of (\ref{eqn:toTeleA}) and naturality of (\ref{eqn:tev}) imply that there is a commutative diagram
\begin{equation}\label{eqn:telescope-diag}
 \xymatrix{
\mrm{R}\opn{Hom}_A(\mrm{H}^0(A),E) \otimes_A \opn{Tel}(A;\mathbf{a}) \ar[r]\ar[d] & \mrm{R}\opn{Hom}_A(\mrm{H}^0(A),E)\ar[d]\\
\mrm{R}\opn{Hom}_A(\mrm{H}^0(A),E \otimes_A \opn{Tel}(A;\mathbf{a})) \ar[r]       & \mrm{R}\opn{Hom}_A(\mrm{H}^0(A),E)
}
\end{equation}
in $\cat{D}(A)$. 
Since $\uflatdim_A(\opn{Tel}(A;\mathbf{a}) < \infty$, 
by Proposition \ref{prop:eval} the leftmost vertical map is an isomorphism.
The rightmost vertical map is the identity so it is also an isomorphism.
By adjunction and the base change property of the telescope DG-module,
the top horizontal map of (\ref{eqn:telescope-diag}) is the same as the map
\[
\mrm{R}\opn{Hom}_A(\mrm{H}^0(A),E) \otimes_{\mrm{H}^0(A)} \opn{Tel}(\mrm{H}^0(A);\bar{\mathbf{a}}) \to \mrm{R}\opn{Hom}_A(\mrm{H}^0(A),E)
\]
induced from the map $\opn{Tel}(\mrm{H}^0(A);\bar{\mathbf{a}}) \to \mrm{H}^0(A)$ of (\ref{eqn:toTele}).
Hence, by (\ref{eqn:diag-of-rgamma}), if the map 
\[
\mrm{R}\Gamma_{\a}\left(\mrm{R}\opn{Hom}_A(\mrm{H}^0(A),E)\right) \to \mrm{R}\opn{Hom}_A(\mrm{H}^0(A),E)
\]
is an isomorphism then all maps in (\ref{eqn:telescope-diag}) are isomorphisms.
In particular, if $\a \subseteq \p$, this is the case since by definition $\mrm{R}\opn{Hom}_A(\mrm{H}^0(A),E) = E(\mrm{H}^0(A),\p)$.
We deduce by Proposition \ref{prop:RHomReflect} that if $\a \subseteq \p$ then the map $E \otimes_A \opn{Tel}(A;\mathbf{a}) \to E$ is an isomorphism,
and hence by (\ref{eqn:RGammaIso}), the map $\mrm{R}\Gamma_{\a}(E) \to E$ is an isomorphism.

On the other hand, if $\a \nsubseteq \p$, then since $\mrm{R}\Gamma_{\a}(E(\mrm{H}^0(A),\p)) = 0$, 
we see that 
\[
\mrm{R}\opn{Hom}_A(\mrm{H}^0(A),E \otimes_A \opn{Tel}(A;\mathbf{a})) \cong 0.
\]
Hence, by Proposition \ref{prop:RHomInf} we deduce that 
\[
0 \cong E \otimes_A \opn{Tel}(A;\mathbf{a}) \cong \mrm{R}\Gamma_{\a}(E),
\]
proving the result.
\end{proof}

\begin{prop}\label{prop:rgammacol}
Let $A$ be a commutative DG-ring,
and let $\a\subseteq \mrm{H}^0(A)$ be a finitely generated ideal.
Then the local cohomology functor 
\[
\mrm{R}\Gamma_{\a}:\cat{D}(A) \to \cat{D}(A)
\]
commutes with arbitrary direct sums.
\end{prop}
\begin{proof}
This is because $\mrm{R}\Gamma_{\a}$ is left adjoint to $\mrm{L}\Lambda_{\a}$,
so it preserves all colimits, and in particular, direct sums.
\end{proof}

\begin{cor}
Let $A$ be a commutative noetherian DG-ring,
and let $\a\subseteq \mrm{H}^0(A)$ be an ideal.
Given $I \in \INJ(A)$,
one has $\mrm{R}\Gamma_{\a}(I) \in \INJ(A)$.
\end{cor}
\begin{proof}
Given $I \in \INJ(A)$,
by Corollary \ref{cor:sum-of-ind} there is an isomorphism
\[
I \cong \bigoplus_{\alpha \in J} I_{\alpha},
\]
such that each $I_{\alpha}$ is an indecomposable element of $\INJ(A)$.
Hence, by Proposition \ref{prop:rgammacol}, we have
\[
\mrm{R}\Gamma_{\a}(I) \cong \mrm{R}\Gamma_{\a}\left(\bigoplus_{\alpha \in J} I_{\alpha}\right) = \bigoplus_{\alpha \in J} \left( \mrm{R}\Gamma_{\a}( I_{\alpha})\right)
\]
By Theorem \ref{thm:bass-papp}, it is thus enough to show that 
\[
\mrm{R}\Gamma_{\a}( I_{\alpha}) \in \INJ(A)
\]
for each $\alpha \in J$, and this follows from Proposition \ref{prop:lcofinj}.
\end{proof}

\begin{rem}
This corollary is essentially a particular case of \cite[Theorem 3.5]{Sh3}.
That theorem states more generally that the inequality
\[
\injdim_A(\mrm{R}\Gamma_{\a}(M)) \le \injdim_A(M)
\]
always holds. 
The proof given here is much more conceptual. 
A search for such a conceptual proof lead to the results of this paper.
\end{rem}

Recall that the homotopy category of non-positive DG-rings,
which we will denote by $\cat{Ho}(\cat{DGR})$ is obtained from the category of non-positive DG-rings by formally inverting quasi-isomorphisms.
This is the homotopy category of a Quillen model category, 
but we will not need the model structure in this paper.

Given a DG-ring $A$ and $M \in \cat{D}(A)$,
the derived endomorphism DG-ring of $M$, 
denoted by $\mrm{R}\opn{Hom}_A(M,M)$ is the DG-ring defined as follows:
let $P \to M$ be a K-projective resolution of $M$,
and define
\[
\mrm{R}\opn{Hom}_A(M,M) := \opn{Hom}_A(P,P).
\]
The latter has naturally the structure of a DG-ring, 
where multiplication is given by composition.
According to \cite[Proposition 3.3]{PSY3},
this DG-ring is independent of the chosen resolution up to isomorphism in $\cat{Ho}(\cat{DGR})$.

\begin{prop}\label{prop:EisOverHat}
Let $(A,\m)$ be a commutative noetherian local DG-ring,
and let $\widehat{A} := \mrm{L}\Lambda(A,\m)$ be its derived $\m$-adic completion.
Denote by $\widehat{\m}$ the maximal ideal of $\mrm{H}^0(\widehat{A})$,
and by $Q:\cat{D}(\widehat{A}) \to \cat{D}(A)$ the forgetful functor.
\begin{enumerate}
\item There is an isomorphism
\[
Q(E(\widehat{A},\widehat{\m})) \cong E(A,\m).
\]
in $\cat{D}(A)$.
\item There is an isomorphism in $\cat{Ho}(\cat{DGR})$ between the following derived endomorphism DG-rings:
\[
\mrm{R}\opn{Hom}_A(E(A,\m),E(A,\m)) \cong \mrm{R}\opn{Hom}_{\widehat{A}}(E(\widehat{A},\widehat{\m}),E(\widehat{A},\widehat{\m})).
\]
\end{enumerate}
\end{prop}
We will prove this proposition here under the additional assumption that $A$ has bounded cohomology.
A proof without this assumption will be given in Appendix \ref{theAppendix}. 
The reason for this is that only in the bounded case we are able to give the morally correct proof of this fact,
and the proof of the more general result is quite technical.
\begin{proof}
(1): Replacing $A$ by a quasi-isomorphic DG-ring if necessary as in (\ref{eqn:completion-diag}),
we may assume that $A \to \widehat{A}$ is a map of DG-rings (and not only in the homotopy category of DG-rings).
Let $E:=\mrm{R}\widehat{\Gamma}_{\m}(E(A,\m)) \in \cat{D}^{+}(\widehat{A})$.
Note that
\[
Q(E) = \mrm{R}\Gamma_{\m}(E(A,\m)) \cong E(A,\m)
\]
where the isomorphism is by Proposition \ref{prop:lcofinj}.
It remains to show that $E \cong E(\widehat{A},\widehat{\m})$.

Under the assumption that $A$ has bounded cohomology, 
it follows by \cite[Corollary 4.6]{Sh3} that $\bflatdim_A(\widehat{A}) = 0$.
Hence, by \cite[Lemma 6.3]{Sh1}, we have that
\[
\mrm{H}^0(\widehat{A}) \cong \mrm{H}^0(A)\otimes^{\mrm{L}}_A \widehat{A}.
\]
Hence, by adjunction:
\begin{equation}\label{eqn:rhomovercomp}
\mrm{R}\opn{Hom}_{\widehat{A}}(\mrm{H}^0(\widehat{A}),E) \cong
\mrm{R}\opn{Hom}_A(\mrm{H}^0(A),E) = \mrm{R}\opn{Hom}_A(\mrm{H}^0(A),E(A,\m)).
\end{equation}
The latter is isomorphic by definition to $E(\mrm{H}^0(A),\m)$, so in particular it is concentrated in degree $0$.
Hence, $\mrm{R}\opn{Hom}_{\widehat{A}}(\mrm{H}^0(\widehat{A}),E)$ is also concentrated in degree $0$.
Applying $\mrm{H}^0$ to (\ref{eqn:rhomovercomp}) and using the fact that 
\[
E(\mrm{H}^0(A),\m) = E(\mrm{H}^0(\widehat{A}),\widehat{\m})
\]
is a $\mrm{H}^0(\widehat{A})$-module,
we obtain a sequence of $\mrm{H}^0(\widehat{A})$-linear isomorphisms,
which imply that
\[
\mrm{H}^0\left(\mrm{R}\opn{Hom}_{\widehat{A}}(\mrm{H}^0(\widehat{A}),E)\right) \cong E(\mrm{H}^0(\widehat{A}),\widehat{\m}).
\]
Since this is an injective $\mrm{H}^0(\widehat{A})$-module,
we deduce from Proposition \ref{prop:INJCHAR} that $E \in \INJ(\widehat{A})$,
and hence $E \cong E(\widehat{A},\widehat{\m})$.

(2): By Proposition \ref{prop:lcofinj}
we know that $E(A,\m) \cong \mrm{R}\Gamma_{\m}(E(A,\m))$.
Hence, by \cite[Lemma 5.3]{Sh2}, we have an isomorphism
\[
E(A,\m) \otimes^{\mrm{L}}_A \widehat{A} \cong E(\widehat{A},\widehat{\m})
\]
in $\cat{D}(\widehat{A})$.
Hence, the DG-ring map
\[
\mrm{R}\opn{Hom}_A(E(A,\m),E(A,\m)) \to \mrm{R}\opn{Hom}_{\widehat{A}}(E(A,\m)\otimes^{\mrm{L}}_A \widehat{A},E(A,\m)\otimes^{\mrm{L}}_A \widehat{A}) 
\]
given by $f \mapsto f\otimes^{\mrm{L}}_A \widehat{A}$ is a quasi-isomorphism,
and as the latter is quasi-isomorphic to the DG-ring
\[
\mrm{R}\opn{Hom}_{\widehat{A}}(E(\widehat{A},\widehat{\m}),E(\widehat{A},\widehat{\m})),
\]
we are done.
\end{proof}

Recall that given a commutative noetherian DG-ring $A$,
a dualizing DG-module $R$ over $A$ is an element $R \in \cat{D}^{+}(A)$ such that $\injdim_A(R) < \infty$,
and such that the natural map $A \to \mrm{R}\opn{Hom}_A(R,R)$ is an isomorphism in $\cat{D}(A)$.
If $A$ is a ring, this coincides with Grothendieck's definition of a dualizing complex.
If $R$ is a dualizing DG-module over $A$ then the natural map
\[
M \to \mrm{R}\opn{Hom}_A(\mrm{R}\opn{Hom}_A(M,R),R)
\]
is an isomorphism for all $M \in \cat{D}_{\mrm{f}}(A)$.

The next result is a DG-version of \cite[Proposition 4.3.8]{Lu}:
\begin{prop}\label{prop:liftdual}
Let $A$ be a commutative noetherian DG-ring,
and let $R \in \cat{D}^{+}(A)$.
If $\mrm{R}\opn{Hom}_A(\mrm{H}^0(A),R)$ is a dualizing complex over $\mrm{H}^0(A)$,
then $R$ is a dualizing DG-module over $A$.
\end{prop}
\begin{proof}
Theorem \ref{thm:injDG} implies that $\injdim_A(R) < \infty$.
Given $M \in \cat{D}_{\mrm{f}}(A)$ such that $\amp(M) = 0$,
there is a finitely generated $\mrm{H}^0(A)$-module $N$ and an isomorphism
$M \cong N$ in $\cat{D}(A)$.
But then, by adjunction, the morphism
\[
N \to \mrm{R}\opn{Hom}_A(\mrm{R}\opn{Hom}_A(N,R),R)
\]
is an isomorphism if and only if the morphism
\[
N \to \mrm{R}\opn{Hom}_{\mrm{H}^0(A)}(\mrm{R}\opn{Hom}_{\mrm{H}^0(A)}(N,\mrm{R}\opn{Hom}_A(\mrm{H}^0(A),R)),\mrm{R}\opn{Hom}_A(\mrm{H}^0(A),R))
\]
is an isomorphism, and that is true because by assumption $\mrm{R}\opn{Hom}_A(\mrm{H}^0(A),R)$ is a dualizing complex over $\mrm{H}^0(A)$.
Hence, the morphism
\[
M \to \mrm{R}\opn{Hom}_A(\mrm{R}\opn{Hom}_A(M,R),R)
\]
is also an isomorphism.
Since both of the functors $1_{\cat{D}(A)}$ and
\[
\mrm{R}\opn{Hom}_A(\mrm{R}\opn{Hom}_A(-,R),R)
\]
have finite cohomological dimension, 
we deduce by the lemma on way-out functors that the map
\[
M \to \mrm{R}\opn{Hom}_A(\mrm{R}\opn{Hom}_A(M,R),R)
\]
is an isomorphism for all $M \in \cat{D}_{\mrm{f}}(A)$,
so that $R$ is a dualizing DG-module over $A$.
\end{proof}

Before the next result, we must recall the theory of t-dualizing complexes.
Let $(A,\mfrak{m})$ be a complete noetherian local ring.
A complex $M \in \cat{D}(A)$ is called cohomologically $\mfrak{m}$-adically cofinite if $M \cong \mrm{R}\Gamma_{\mfrak{m}}(N)$
for some $N \in \cat{D}^{\mrm{b}}_{\mrm{f}}(A)$. See \cite[Section 3]{PSY2} for details about this concept. 
Given a cohomologically $\mfrak{m}$-adically cofinite complex $M$ over $A$,
one says that $M$ is a t-dualizing complex if $\injdim_A(M) < \infty$,
and the canonical map $A \to \mrm{R}\opn{Hom}_A(M,M)$ is an isomorphism in $\cat{D}(A)$.
This terminology is from \cite[Section 2.5]{AJL}.
This concept was first defined and studied in \cite[Section 5]{Ye3}. 
If $M$ is a t-dualizing complex over $A$,
then $\mrm{L}\Lambda_{\mfrak{m}}(M)$ is a dualizing complex over $A$.
Conversely, if $N$ is a dualizing complex over $A$ then $\mrm{R}\Gamma_{\mfrak{m}}(N)$ is a t-dualizing complex over $A$.
With this terminology in hand, 
it is an immediate corollary of Grothendieck's local duality theorem (\cite[Theorem V.6.2]{RD})
that if $(A,\mfrak{m})$ is a complete noetherian local ring then $E(A,\mfrak{m})$ is a t-dualizing complex over $A$.

It is a corollary of the Cohen Structure theorem that every complete noetherian local ring has dualizing complexes.
Similarly, we have:

\begin{prop}\label{prop:hasDual}
Let $(A,\m)$ be a commutative noetherian local DG-ring,
and assume that $A$ is cohomologically $\m$-adically complete.
Let $E = E(A,\m) \in \INJ(A)$.
Then the DG-module $R := \mrm{L}\Lambda_{\m}(E)$ is a dualizing DG-module over $A$.
\end{prop}
The point is that $E$ is a t-dualizing DG-module over $A$, and hence $R$ is a dualizing DG-module over $A$.
As the concept of a t-dualizing DG-module over a commutative adic DG-ring was not yet defined and studied, 
let us prove this directly.
\begin{proof}
By \cite[Proposition 2.17]{Sh2}, we have that
\[
\mrm{R}\opn{Hom}_A(\mrm{H}^0(A),R) = \mrm{R}\opn{Hom}_A\left(\mrm{H}^0(A),\mrm{L}\Lambda^A_{\m}(E)\right) \cong
\mrm{L}\Lambda^{\mrm{H}^0(A)}_{\m}\left(\mrm{R}\opn{Hom}_A(\mrm{H}^0(A),E)\right)
\]
Since $\mrm{R}\opn{Hom}_A(\mrm{H}^0(A),E) \cong E(\mrm{H}^0(A),\m)$,
by the discussion preceding this proof 
we see that $\mrm{R}\opn{Hom}_A(\mrm{H}^0(A),R)$ is a dualizing complex over $\mrm{H}^0(A)$.
Hence, the result follows from Proposition \ref{prop:liftdual}.
\end{proof}

\begin{thm}
Let $A$ be a commutative noetherian DG-ring,
and let $\p \in \opn{Spec}(\mrm{H}^0(A))$.
Then there is an isomorphism in $\cat{Ho}(\cat{DGR})$:
\[
\mrm{R}\opn{Hom}_A(E(A,\p),E(A,\p)) \cong \widehat{A_{\p}},
\]
where the right hand side is defined to be the derived $\p$-adic completion of the local DG-ring $A_{\p}$.
\end{thm}
\begin{proof}
By Proposition \ref{prop:localz},
the map
\[
\mrm{R}\opn{Hom}_A(E(A,\p),E(A,\p)) \to \mrm{R}\opn{Hom}_{A_{\p}}(E(A,\p)_{\p},E(A,\p)_{\p})
\]
induced by the localization morphism $E(A,\p) \to E(A,\p)_{\p}$ is a quasi-isomorphism.
Let us denote by $\widehat{\p}$ the ideal of definition of $\widehat{A_{\p}}$.
By Proposition \ref{prop:EisOverHat}(2),
there is an isomorphism in $\cat{Ho}(\cat{DGR})$:
\[
\mrm{R}\opn{Hom}_{A_{\p}}(E(A,\p)_{\p},E(A,\p)_{\p}) \cong
\mrm{R}\opn{Hom}_{\widehat{A_{\p}}}(E(\widehat{A_{\p}},\widehat{\p}),E(\widehat{A_{\p}},\widehat{\p})).
\]
By Proposition \ref{prop:lcofinj}, we know that
\[
E(\widehat{A_{\p}},\widehat{\p}) \cong \mrm{R}\Gamma_{\widehat{\p}}(E(\widehat{A_{\p}},\widehat{\p})),
\]
so by the Greenlees-May duality, there is an isomorphism in $\cat{Ho}(\cat{DGR})$:
\[
\mrm{R}\opn{Hom}_{\widehat{A_{\p}}}(E(\widehat{A_{\p}},\widehat{\p}),E(\widehat{A_{\p}},\widehat{\p})) \cong
\mrm{R}\opn{Hom}_{\widehat{A_{\p}}}(\mrm{L}\Lambda_{\widehat{\p}}(E(\widehat{A_{\p}},\widehat{\p})),\mrm{L}\Lambda_{\widehat{\p}}(E(\widehat{A_{\p}},\widehat{\p})))
\]
But by Proposition \ref{prop:hasDual},
the DG-module $\mrm{L}\Lambda_{\widehat{\p}}(E(\widehat{A_{\p}},\widehat{\p}))$ is a dualizing DG-module over $\widehat{A_{\p}}$,
which implies that the canonical map
\[
\widehat{A_{\p}} \to \mrm{R}\opn{Hom}_{\widehat{A_{\p}}}(\mrm{L}\Lambda_{\widehat{\p}}(E(\widehat{A_{\p}},\widehat{\p})),\mrm{L}\Lambda_{\widehat{\p}}(E(\widehat{A_{\p}},\widehat{\p})))
\]
is an isomorphism. 
Combining all these isomorphisms (which do not all go in the same direction, hence the need for the homotopy category),
we obtain the required isomorphism in $\cat{Ho}(\cat{DGR})$, proving the result.
\end{proof}

Let $(A,\m,\k)$ be a noetherian local DG-ring,
and assume that $A$ has a dualizing DG-module $R$.
By \cite[Proposition 7.5]{Ye2},
this implies that $\mrm{R}\opn{Hom}_A(\mrm{H}^0(A),R)$ is a dualizing complex over the local noetherian ring $\mrm{H}^0(A)$.
Hence, by \cite[Proposition V.3.4]{RD},
there is exactly one integer $d$ such that 
\[
\Ext^d_{\mrm{H}^0(A)}(\k,\mrm{R}\opn{Hom}_A(\mrm{H}^0(A),R)) \cong \k.
\]
Following \cite[Section V.6]{RD}, we say that $R$ is a \textbf{normalized dualizing DG-module} if $d=0$.
Of course even if $R$ is not normalized, there is always some shift of $R$ which is normalized.

The above discussion also essentially follows from \cite[Theorem 3.2]{FIJ},
but the authors of that paper make the additional assumption that $A^0$ is a noetherian ring.

\begin{prop}\label{lem:rgammaofrhom}
Let $A$ be a commutative noetherian DG-ring,
and let $\a\subseteq \mrm{H}^0(A)$ be an ideal.
Then for any $M \in \cat{D}^{+}(A)$,
there is a natural isomorphism
\[
\mrm{R}\Gamma_{\a}^{\mrm{H}^0(A)} \left(\RHom_A(\mrm{H}^0(A),M)\right) \cong \RHom_A\left(\mrm{H}^0(A),\mrm{R}\Gamma_{\a}^{A}(M)\right)
\]
in $\cat{D}(\mrm{H}^0(A))$.
\end{prop}
\begin{proof}
This was shown in \cite[Lemma 3.4]{Sh3}.
There we also assumed that $M \in \cat{D}^{\mrm{b}}(A)$,
but using Proposition \ref{prop:eval} the same proof works under this weaker assumption.
\end{proof}

\begin{prop}\label{prop:rgamhomfinite}
Let $A$ be a commutative noetherian DG-ring,
and let $\a\subseteq \mrm{H}^0(A)$ be an ideal.
Then for any $M \in \cat{D}^{-}_{\mrm{f}}(A)$,
and any $N \in \cat{D}^{+}(A)$,
there is a bifunctorial isomorphism
\[
\mrm{R}\opn{Hom}_A(M,\mrm{R}\Gamma_{\a}(N)) \cong \mrm{R}\Gamma_{\a}\left(\mrm{R}\opn{Hom}_A(M,N)\right)
\]
in $\cat{D}(A)$.
\end{prop}
\begin{proof}
Let $\mathbf{a}$ be a finite sequence of elements of $A^0$ whose image in $\mrm{H}^0(A)$ generates $\a$.
Then by (\ref{eqn:diag-of-rgamma}), we have that
\[
\mrm{R}\opn{Hom}_A(M,\mrm{R}\Gamma_{\a}(N)) \cong \mrm{R}\opn{Hom}_A(M,N\otimes_A \opn{Tel}(A;\mathbf{a})).
\]
By Proposition \ref{prop:eval},
we have a natural isomorphism
\[
\mrm{R}\opn{Hom}_A(M,N\otimes_A \opn{Tel}(A;\mathbf{a})) \cong \mrm{R}\opn{Hom}_A(M,N)\otimes_A \opn{Tel}(A;\mathbf{a}),
\]
so the result follows from applying (\ref{eqn:diag-of-rgamma}) again.
\end{proof}

\begin{prop}\label{prop:rgammaofdual}
Let $(A,\m,\k)$ be a noetherian local DG-ring,
and let $R$ be a normalized dualizing DG-module over $A$.
Then
\[
\mrm{R}\Gamma_{\m}(R) \cong E(A,\m).
\]
\end{prop}
\begin{proof}
By Proposition \ref{lem:rgammaofrhom},
\[
\mrm{R}\opn{Hom}_A(\mrm{H}^0(A),\mrm{R}\Gamma^A_{\m}(R)) \cong \mrm{R}\Gamma^{\mrm{H}^0(A)}_{\m}(\mrm{R}\opn{Hom}_A(\mrm{H}^0(A),R)),
\]
and by the normalization assumption and \cite[Proposition V.6.1]{RD},
the latter is isomorphic to $E(\mrm{H}^0(A),\m)$,
so Proposition \ref{prop:INJCHAR} implies the result.
\end{proof}

The next result is a DG-generalization of Grothendieck's local duality:

\begin{thm}\label{thm:localdual}
Let $(A,\m,\k)$ be a local noetherian DG-ring,
and let $R$ be a normalized dualizing DG-module over $A$.
Let $E := E(A,\m) \in \INJ(A)$.
\begin{enumerate}
\item For every $M \in \cat{D}^{+}_{\mrm{f}}(A)$,
there is an isomorphism
\[
\mrm{R}\Gamma_{\m}(M) \cong \mrm{R}\opn{Hom}_A(\mrm{R}\opn{Hom}_A(M,R),E)
\]
in $\cat{D}(A)$ which is functorial in $M$.
\item
For every $M \in \cat{D}^{+}_{\mrm{f}}(A)$, and any $n \in \mathbb{Z}$, there is an isomorphism
\[
\mrm{H}_{\m}^n(M) := \mrm{H}^n\left(\mrm{R}\Gamma_{\m}(M)\right) \cong \opn{Hom}_{\mrm{H}^0(A)}\left(\Ext^{-n}_A(M,R),\mrm{H}^0(E)\right)
\]
in $\opn{Mod}(\mrm{H}^0(A))$ which is functorial in $M$.
\end{enumerate}
\end{thm}
\begin{proof}
By Proposition \ref{prop:rgammaofdual}, we have that
\[
\mrm{R}\opn{Hom}_A(\mrm{R}\opn{Hom}_A(M,R),E) \cong \mrm{R}\opn{Hom}_A(\mrm{R}\opn{Hom}_A(M,R),\mrm{R}\Gamma_{\m}(R)).
\]  
Since $M \in \cat{D}^{+}_{\mrm{f}}(A)$, 
by \cite[Proposition 7.2]{Ye2},
we have that $\mrm{R}\opn{Hom}_A(M,R) \in \cat{D}^{-}_{\mrm{f}}(A)$.
Hence, by Proposition \ref{prop:rgamhomfinite}, we obtain:
\[
\mrm{R}\opn{Hom}_A(\mrm{R}\opn{Hom}_A(M,R),\mrm{R}\Gamma_{\m}(R)) \cong
\mrm{R}\Gamma_{\m}\left(\mrm{R}\opn{Hom}_A(\mrm{R}\opn{Hom}_A(M,R),R)\right).
\]
Since $R$ is a dualizing DG-module,
the latter is naturally isomorphic to $\mrm{R}\Gamma_{\m}(M)$, proving (1).
To obtain (2), simply apply the functor $\mrm{H}^n$ to (1),
and use Theorem \ref{thm:coho-of-rhom} which holds because $E \in \INJ(A)$.
\end{proof}

\begin{rem}
The recent paper \cite{BHV} made a detailed study of local duality in an abstract framework.
It is not clear to us if the above result which is a concrete result in a rather abstract setup can be deduced from the results of that paper.
\end{rem}

\begin{cor}
Let $(A,\m,\k)$ be a local noetherian DG-ring which has a dualizing DG-module.
Then for every $M \in \cat{D}^{+}_{\mrm{f}}(A)$ and any $n \in \mathbb{Z}$,
the $\mrm{H}^0(A)$-module $\mrm{H}_{\m}^n(M)$ is artinian.
\end{cor}

\begin{cor}
Let $(A,\m,\k)$ be a local noetherian DG-ring with bounded cohomology,
and let $R$ be a dualizing DG-module over $A$.
Then 
\[
\amp\left(\mrm{R}\Gamma_{\m}(A)\right) = \amp(R).
\]
\end{cor}
\begin{proof}
Since the amplitude of a DG-module does not change under the translation functor, 
we may assume without loss of generality that $R$ is a normalized dualizing DG-module.
The assumption that $A$ has bounded cohomology implies that $A \in \cat{D}^{+}_{\mrm{f}}(A)$.
Hence, by the local duality theorem,
for each $n \in \mathbb{Z}$ we have that
\[
\mrm{H}^n\left(\mrm{R}\Gamma_{\m}(A)\right) \cong \opn{Hom}_{\mrm{H}^0(A)}(\mrm{H}^{-n}(R),\mrm{H}^0(E)).
\]
Since $\mrm{H}^0(E)$ is a cogenerator of $\opn{Mod}(\mrm{H}^0(A))$,
we deduce that $\mrm{H}^n\left(\mrm{R}\Gamma_{\m}(A)\right) = 0$ if and only if $\mrm{H}^{-n}(R) = 0$,
hence the equality.
\end{proof}

\appendix

\section{Proposition \ref{prop:EisOverHat} in the unbounded case}\label{theAppendix}

The purpose of this appendix is to prove Proposition \ref{prop:EisOverHat} without the assumption that $\mrm{H}(A)$ is bounded.
Let us recall the statement:

\begin{prop}
Let $(A,\m)$ be a commutative noetherian local DG-ring,
and let $\widehat{A} := \mrm{L}\Lambda(A,\m)$ be its derived $\m$-adic completion.
Denote by $\widehat{\m}$ the maximal ideal of $\mrm{H}^0(\widehat{A})$,
and by $Q:\cat{D}(\widehat{A}) \to \cat{D}(A)$ the forgetful functor.
\begin{enumerate}
\item There is an isomorphism
\[
Q(E(\widehat{A},\widehat{\m})) \cong E(A,\m).
\]
in $\cat{D}(A)$.
\item There is an isomorphism in $\cat{Ho}(\cat{DGR})$ between the following derived endomorphism DG-rings:
\[
\mrm{R}\opn{Hom}_A(E(A,\m),E(A,\m)) \cong \mrm{R}\opn{Hom}_{\widehat{A}}(E(\widehat{A},\widehat{\m}),E(\widehat{A},\widehat{\m})).
\]
\end{enumerate}
\end{prop}

We shall need the following lemma.
\begin{lem}\label{lem:rgforget}
Let $A$ be a commutative noetherian DG-ring,
and let $\a\subseteq \mrm{H}^0(A)$ be an ideal.
Let $\widehat{A} := \mrm{L}\Lambda(A,\a)$, and let
$\widehat{\a} := \a \cdot \mrm{H}^0(\widehat{A})$.
Denoting by $Q: \cat{D}(\widehat{A}) \to \cat{D}(A)$ the forgetful functor,
there is an isomorphism
\[
Q\circ \mrm{R}\Gamma_{\widehat{\a}}(-) \cong \mrm{R}\Gamma_{\a}\circ Q(-)
\]
of functors $\cat{D}(\widehat{A}) \to \cat{D}(A)$.
\end{lem}
\begin{proof}
Replacing $A$ by a quasi-isomorphic DG-ring if needed, 
we may assume that $A \to \widehat{A}$ is a DG-ring map.
Let us denote it by $\tau$.
Let $\mathbf{a}$ be a finite sequence of elements of $A^0$ whose image in $\mrm{H}^0(A)$ generates $\a$.
Then the commutative diagram
\[
\xymatrix{
A^0 \ar[r]^{\tau^0}\ar[d] & \widehat{A}^0\ar[d]\\
\mrm{H}^0(A) \ar[r]& \mrm{H}^0(\widehat{A})
}
\]
implies that the image of $\tau^0(\mathbf{a})$ in $\mrm{H}^0(\widehat{A})$ generates $\widehat{\a}$.
Hence, given $M \in \cat{D}(\widehat{A})$, by (\ref{eqn:RGammaIso}) we have:
\[
Q(\mrm{R}\Gamma_{\widehat{\a}}(M)) \cong
Q(\opn{Tel}(A;\mathbf{a})\otimes_A \widehat{A} \otimes_{\widehat{A}} M) \cong
\opn{Tel}(A;\mathbf{a})\otimes_A Q(M),
\]
so using (\ref{eqn:RGammaIso}) we obtain the result.
\end{proof}

We are now ready to prove Proposition \ref{prop:EisOverHat}(1):

\begin{proof}
As in our first proof of Proposition \ref{prop:EisOverHat}, we may assume that $A \to \widehat{A}$ is a map of DG-rings.
Let 
\[
E:= \mrm{R}\Gamma^{\widehat{A}}_{\widehat{\m}} \left(\mrm{L}\widehat{\Lambda}_{\m}(E(A,\m))\right) \in \cat{D}^{+}(\widehat{A}).
\]
By Lemma \ref{lem:rgforget}, we have:
\[
Q(E) \cong \mrm{R}\Gamma^A_{\m}(Q((\mrm{L}\widehat{\Lambda}_{\m}(E(A,\m))))).
\]
From this and (\ref{eqn:whlambda}) we obtain an isomorphism:
\[
Q(E) \cong \mrm{R}\Gamma^A_{\m}\mrm{L}\Lambda^A_{\m}(E(A,\m)),
\]
and by the MGM equivalence,
we have $Q(E) \cong \mrm{R}\Gamma^A_{\m}(E(A,\m))$,
so Proposition \ref{prop:lcofinj} implies that $Q(E) \cong E(A,\m)$.
It remains to show that $E \cong E(\widehat{A},\widehat{\m})$.
To do this, we wish to calculate $\mrm{R}\opn{Hom}_{\widehat{A}}(\mrm{H}^0(\widehat{A}),E)$.
By Proposition \ref{lem:rgammaofrhom} we have:
\[
\mrm{R}\opn{Hom}_{\widehat{A}}(\mrm{H}^0(\widehat{A}),E) \cong
\mrm{R}\Gamma^{\mrm{H}^0(\widehat{A})}_{\widehat{\m}}\left( \mrm{R}\opn{Hom}_{\widehat{A}}(\mrm{H}^0(\widehat{A}),\mrm{L}\widehat{\Lambda}_{\m}(E(A,\m)))\right).
\]
By \cite[Proposition 5.8]{Sh2}, 
there is an isomorphism
\[
\mrm{R}\opn{Hom}_{\widehat{A}}(\mrm{H}^0(\widehat{A}),\mrm{L}\widehat{\Lambda}_{\m}(E(A,\m))) \cong
\mrm{R}\opn{Hom}_A(\mrm{H}^0(\widehat{A}),\mrm{L}\Lambda_{\m}(E(A,\m)))
\]
in $\cat{D}(\mrm{H}^0(\widehat{A}))$.
Let $Q^0:\cat{D}(\mrm{H}^0(\widehat{A})) \to \cat{D}(\mrm{H}^0(A))$ be the forgetful functor.
Then using Lemma \ref{lem:rgforget} and the fact that because $A$ is noetherian it holds that $\mrm{H}^0(\widehat{A})$ is the $\m$-adic completion of $\mrm{H}^0(A)$, we have:
\[
Q^0\left(\mrm{R}\opn{Hom}_{\widehat{A}}(\mrm{H}^0(\widehat{A}),E)\right) \cong
\mrm{R}\Gamma^{\mrm{H}^0(A)}_{\m}\left( 
\mrm{R}\opn{Hom}_A(\mrm{H}^0(\widehat{A}),\mrm{L}\Lambda_{\m}(E(A,\m)))
\right)
\]
Since there is an isomorphism $\mrm{L}\Lambda_{\m}(\mrm{H}^0(A)) \cong \mrm{H}^0(\widehat{A})$,
we deduce by the Greenlees-May duality that
\begin{eqnarray}
\mrm{R}\Gamma^{\mrm{H}^0(A)}_{\m}\left( 
\mrm{R}\opn{Hom}_A(\mrm{H}^0(\widehat{A}),\mrm{L}\Lambda_{\m}(E(A,\m)))
\right)
\cong\nonumber\\
\mrm{R}\Gamma^{\mrm{H}^0(A)}_{\m}\left( 
\mrm{R}\opn{Hom}_A(\mrm{H}^0(A),\mrm{L}\Lambda_{\m}(E(A,\m)))
\right)\nonumber
\end{eqnarray}
By \cite[Proposition 2.17]{Sh2},
there is an isomorphism
\begin{eqnarray}
\mrm{R}\Gamma^{\mrm{H}^0(A)}_{\m}\left( 
\mrm{R}\opn{Hom}_A(\mrm{H}^0(A),\mrm{L}\Lambda_{\m}(E(A,\m)))
\right) \cong\nonumber\\
\mrm{R}\Gamma^{\mrm{H}^0(A)}_{\m} \mrm{L}\Lambda^{\mrm{H}^0(A)}_{\m} \left( 
\mrm{R}\opn{Hom}_A(\mrm{H}^0(A),E(A,\m))
\right),\nonumber
\end{eqnarray}
and by the MGM equivalence, 
the latter is isomorphic to 
\[
\mrm{R}\Gamma^{\mrm{H}^0(A)}_{\m} \left( 
\mrm{R}\opn{Hom}_A(\mrm{H}^0(A),E(A,\m))\right) = E(\mrm{H}^0(A),\m).
\]

We have thus shown that
\[
Q^0\left(\mrm{R}\opn{Hom}_{\widehat{A}}(\mrm{H}^0(\widehat{A}),E)\right) \cong E(\mrm{H}^0(A),\m).
\]
Hence, for all $i \ne 0$, 
we have that
\[
\mrm{H}^i\left(\mrm{R}\opn{Hom}_{\widehat{A}}(\mrm{H}^0(\widehat{A}),E)\right) = 0,
\]
while for $i=0$ we have
\[
Q^0(\mrm{H}^0\left(\mrm{R}\opn{Hom}_{\widehat{A}}(\mrm{H}^0(\widehat{A}),E)\right)) = E(\mrm{H}^0(A),\m) = E(\mrm{H}^0(\widehat{A}),\widehat{\m}).
\]
A-priori, we only know that the isomorphism we have between the $\mrm{H}^0(\widehat{A})$-modules
\[
\mrm{H}^0\left(\mrm{R}\opn{Hom}_{\widehat{A}}(\mrm{H}^0(\widehat{A}),E)\right)
\]
and $E(\mrm{H}^0(\widehat{A}),\widehat{\m})$ is $\mrm{H}^0(A)$-linear,
but since these are $\m$-torsion modules, the isomorphism is automatically $\mrm{H}^0(\widehat{A})$-linear.
Hence, we see that
\[
\mrm{R}\opn{Hom}_{\widehat{A}}(\mrm{H}^0(\widehat{A}),E) \cong E(\mrm{H}^0(\widehat{A}),\widehat{\m}),
\]
which implies that $E \cong E(A,\widehat{\m})$, as claimed.
\end{proof}	

\begin{rem}
Let $A$ be a commutative noetherian ring,
let $\mfrak{a}\subseteq A$ be an ideal,
and let $\widehat{A}:=\Lambda_{\mfrak{a}}(A)$.
Denote by $\widehat{\mfrak{a}} := \mfrak{a}\cdot \widehat{A}$ the ideal of definition of $\widehat{A}$.
By \cite[Theorem 2.4]{ShAdic},
there is an isomorphism 
\[
\mrm{R}\Gamma_{\widehat{\mfrak{a}}}\mrm{L}\widehat{\Lambda}_{\mfrak{a}}(-) \cong \mrm{R}\widehat{\Gamma}_{\mfrak{a}}(-)
\]
of functors $\cat{D}(A) \to \cat{D}(\widehat{A})$.

If this result generalizes to commutative noetherian DG-rings,
then in the above proof we will simply get
\[
E \cong \mrm{R}\widehat{\Gamma}_{\m}(E(A,\m)),
\]
as in our original proof of Proposition \ref{prop:EisOverHat}.
\end{rem}

The next lemma is a DG-version of \cite[Theorem 4.3]{SW}.

\begin{lem}\label{lem:llambdad}
Let $A$ be a commutative noetherian DG-ring,
and let $\a\subseteq \mrm{H}^0(A)$ be an ideal.
Let $\widehat{\a}$ be the ideal of definition of $\widehat{A}$.
Then there is an isomorphism
\[
\mrm{L}\Lambda_{\widehat{\a}} \circ \mrm{L}\widehat{\Lambda}_{\a}(-) \cong \mrm{L}\widehat{\Lambda}_{\a}(-).
\]
of functors $\cat{D}(A) \to \cat{D}(\widehat{A})$.
\end{lem}
\begin{proof}
Replacing $A$ by a quasi-isomorphic DG-ring if necessary, 
we may assume that $A \to \widehat{A}$ is a map of DG-rings.
Proceeding as in the proof of Lemma \ref{lem:rgforget},
given $M \in \cat{D}(A)$, 
by \cite[Lemma 5.3]{Sh2},
we have:
\[
\mrm{L}\Lambda_{\widehat{\a}} \circ \mrm{L}\widehat{\Lambda}_{\a}(M) \cong
\mrm{R}\opn{Hom}_{\widehat{A}}(\opn{Tel}(\widehat{A};\tau^0(\mathbf{a})),\mrm{R}\opn{Hom}_A(\widehat{A},\mrm{L}\Lambda_{\a}(M))).
\]
By adjunction and the base change property of the telescope DG-module,
this is isomorphic to
\[
\mrm{R}\opn{Hom}_A(\widehat{A},\mrm{R}\opn{Hom}_A(\opn{Tel}(A;\mathbf{a}),\mrm{L}\Lambda_{\a}(M))),
\]
so the result follows from using \cite[Lemma 5.3]{Sh2} again and the fact that 
\[
\opn{Tel}(A;\mathbf{a}) \cong \opn{Tel}(A;\mathbf{a})\otimes_A \opn{Tel}(A;\mathbf{a}).
\]
\end{proof}
We now prove Proposition \ref{prop:EisOverHat}(2).
\begin{proof}
By our proof of (1), we have that
\begin{eqnarray}
\mrm{R}\opn{Hom}_{\widehat{A}}(E(\widehat{A},\widehat{\m}),E(\widehat{A},\widehat{\m})) = \nonumber\\
\mrm{R}\opn{Hom}_{\widehat{A}}(\mrm{R}\Gamma^{\widehat{A}}_{\widehat{\m}} \left(\mrm{L}\widehat{\Lambda}_{\m}(E(A,\m))\right),\mrm{R}\Gamma^{\widehat{A}}_{\widehat{\m}} \left(\mrm{L}\widehat{\Lambda}_{\m}(E(A,\m))\right)).\nonumber
\end{eqnarray}
By the Greenlees-May duality, 
this is isomorphic to
\[
\mrm{R}\opn{Hom}_{\widehat{A}}(\mrm{L}\Lambda^{\widehat{A}}_{\widehat{\m}} \left(\mrm{L}\widehat{\Lambda}_{\m}(E(A,\m))\right),\mrm{L}\Lambda^{\widehat{A}}_{\widehat{\m}} \left(\mrm{L}\widehat{\Lambda}_{\m}(E(A,\m))\right)).
\]
Lemma \ref{lem:llambdad} implies that the latter is isomorphic to
\[
\mrm{R}\opn{Hom}_{\widehat{A}}(\mrm{L}\widehat{\Lambda}_{\m}(E(A,\m)),\mrm{L}\widehat{\Lambda}_{\m}(E(A,\m))).
\]
According to \cite[Lemma 5.3]{Sh2}, there is an $\widehat{A}$-linear isomorphism
\begin{eqnarray}
\mrm{R}\opn{Hom}_{\widehat{A}}(\mrm{L}\widehat{\Lambda}_{\m}(E(A,\m)),\mrm{L}\widehat{\Lambda}_{\m}(E(A,\m))) \cong\nonumber\\
\mrm{R}\opn{Hom}_{\widehat{A}}(\mrm{L}\widehat{\Lambda}_{\m}(E(A,\m)),\mrm{R}\opn{Hom}_A(\widehat{A},\mrm{L}\Lambda_{\m}(E(A,\m)))).\nonumber
\end{eqnarray}
Hence, by adjunction, we obtain
\begin{eqnarray}
\mrm{R}\opn{Hom}_A(Q(\mrm{L}\widehat{\Lambda}_{\m}(E(A,\m))),\mrm{L}\Lambda_{\m}(E(A,\m))) \cong\nonumber\\ 
\mrm{R}\opn{Hom}_A(\mrm{L}\Lambda_{\m}(E(A,\m)),\mrm{L}\Lambda_{\m}(E(A,\m))),\nonumber
\end{eqnarray}
so the result follows from the Greenlees-May duality and the fact that by Proposition \ref{prop:lcofinj} we have that
\[
E(A,\m) \cong \mrm{R}\Gamma_{\a}(E(A,\m)).
\]
\end{proof}

\textbf{Acknowledgments.}
Work on this paper started as a result of discussions in the 2017 European Talbot Workshop.
The author is thankful to the workshop participants for these discussions.
The author would like to thank Henning Krause for suggesting an alternative proof of Theorem \ref{thm:eqv}.




\begin{thebibliography}{99}

\bibitem{AF}
Avramov, L. L., and Foxby, H. B. (1991). Homological dimensions of unbounded complexes. Journal of Pure and Applied Algebra, 71(2-3), 129-155.

\bibitem{AJL}
Alonso, L., Jeremias A. and Lipman J. (1999). Duality and flat base change on formal schemes. Studies in Duality on Noetherian formal schemes and non-Noetherian ordinary schemes, Contemporary Mathematics, 244, 1-87.


\bibitem{Ba}
Bass, H. (1962). Injective dimension in Noetherian rings. Transactions of the American Mathematical Society, 102(1), 18-29.

\bibitem{BHV}
Barthel, T., Heard, D., \& Valenzuela, G. (2015). Local duality in algebra and topology. arXiv preprint arXiv:1511.03526.

\bibitem{BIK}
Benson, D., Iyengar, S. B., \& Krause, H. (2008). Local cohomology and support for triangulated categories. Ann. Sci. Ec. Norm. Super.(4), 41(4), 573-619.

\bibitem{BN}
B\"{o}kstedt, M., \& Neeman, A. (1993). Homotopy limits in triangulated categories. Compositio Mathematica, 86(2), 209-234.

\bibitem{FIJ}
Frankild, A., Iyengar, S. B., and J{\o}rgensen, P. (2003). Dualizing differential graded modules and Gorenstein differential graded algebras. Journal of the London Mathematical Society, 68(2), 288-306.

\bibitem{FJ}
Frankild, A., \& J{\o}rgensen, P. (2003). Gorenstein differential graded algebras. Israel Journal of Mathematics, 135(1), 327-353.


\bibitem{GM}
Greenlees, J. P. C., and May, J. P. (1992). Derived functors of I-adic completion and local homology. Journal of Algebra, 149(2), 438-453.

\bibitem{GP}
Garkusha, G., \& Prest, M. (2004). Injective objects in triangulated categories. Journal of Algebra and Its Applications, 3(04), 367-389.

\bibitem{RD}
Hartshorne, R. (1966). Residues and duality: lectures notes of a seminar on the work of A. Grothendieck, given at Harvard 1963/64. Springer.

\bibitem{La} Lam, T. Y. (2012). Lectures on modules and rings (Vol. 189). Springer Science \& Business Media.

\bibitem{Lu} Lurie, J. Derived Algebraic Geometry XIV: Representability Theorems, preprint available at \url{http://www.math.harvard.edu/~lurie/papers/DAG-XIV.pdf}.

\bibitem{Lu2} Lurie, J. (2016). Higher algebra. Preprint, available at \url{http://www.math.harvard.edu/~lurie}.

\bibitem{Ma} Matlis, E. (1958). Injective modules over Noetherian rings. Pacific J. Math, 8(3), 511-528.

\bibitem{Ne} Neeman, A. (1996). The Grothendieck duality theorem via Bousfield's techniques and Brown representability. Journal of the American Mathematical Society, 9(1), 205-236.

\bibitem{Pa}
Papp, Z. (1959). On algebraically closed modules. Publ. Math. Debrecen, 6(3), 11-327.

\bibitem{PSY1} 
Porta, M., Shaul, L., \& Yekutieli, A. (2014). On the homology of completion and torsion. Algebras and Representation Theory, 17(1), 31-67.

\bibitem{PSY2}
Porta, M., Shaul, L., \& Yekutieli, A. (2015). Cohomologically cofinite complexes. Communications in Algebra, 43(2), 597-615.

\bibitem{PSY3}
Porta, M., Shaul, L., \& Yekutieli, A. (2014). Completion by derived double centralizer. Algebras and Representation Theory, 17(2), 481-494.

\bibitem{SW}
Sather-Wagstaff, S., \& Wicklein, R. (2016). Extended local cohomology and local homology. Algebras and Representation Theory, 19(5), 1217-1238.


\bibitem{ShHC} Shaul, L. (2016). Hochschild cohomology commutes with adic completion. Algebra \& Number Theory, 10(5), 1001-1029.

\bibitem{Sh2} Shaul, L. (2016). Completion and torsion over commutative DG rings. arXiv preprint arXiv:1605.07447v3.

\bibitem{ShAdic} Shaul, L. (2017). Adic reduction to the diagonal and a relation between cofiniteness and derived completion. Proceedings of the American Mathematical Society, 145(12), 5131-5143.


\bibitem{Sh3} Shaul, L. (2017). Shaul, L. (2018). Homological dimensions of local (co) homology over commutative DG-rings. Canadian Mathematical Bulletin 61, no. 4, 865-877.


\bibitem{Sh1} Shaul, L. (2017). The twisted inverse image pseudofunctor over commutative DG rings and perfect base change. Advances in Mathematics, 320, 279-328.

\bibitem{SS} Schwede, S., \& Shipley, B. (2003). Equivalences of monoidal model categories. Algebraic \& Geometric Topology, 3(1), 287-334.


\bibitem{We}
Weibel, C. A. (1995). An introduction to homological algebra (No. 38). Cambridge university press.

\bibitem{Ye3}
Yekutieli, A. (1998). Smooth formal embeddings and the residue complex. Canadian Journal of Mathematics, 50(4), 863.

\bibitem{Ye2}
Yekutieli, A. (2013). Duality and tilting for commutative DG rings. arXiv preprint arXiv:1312.6411v4.

\bibitem{YeSQ}
Yekutieli, A. (2016). The squaring operation for commutative DG rings. Journal of Algebra, 449, 50-107.

\bibitem{Ye}
Yekutieli, A. (2017). "Derived Categories", book to be published by Cambridge Univ. Press. Preview version: arXiv preprint arXiv:1610.09640v2.


\end{thebibliography}
\end{document}